\documentclass[12pt]{amsart}
\usepackage{amssymb}
\usepackage{amsmath}
\usepackage{amscd}
\usepackage[all]{xy}
\usepackage{tikz}
\usepackage{color}
\usepackage[normalem]{ulem}
\usepackage{mathtools}
\usepackage{cite}
\mathtoolsset{showonlyrefs=true}

\usetikzlibrary{intersections, calc, patterns}

\theoremstyle{plain}
\newtheorem{theorem}{Theorem}[section]
\newtheorem{corollary}[theorem]{Corollary}

\newtheorem{lemma}[theorem]{Lemma}
\newtheorem{proposition}[theorem]{Proposition}

\theoremstyle{definition}

\newtheorem{example}[theorem]{Example}

\newtheorem{assumption}[theorem]{Assumption}

\theoremstyle{remark}
\newtheorem{remark}[theorem]{Remark}

\def\Z{{\mathbb Z}}
\def\N{{\mathbb N}}
\def\R{{\mathbb R}}
\def\C{{\mathbb C}}
\def\Q{{\mathbb Q}}

\def\cG{{\mathcal G}}
\def\cH{{\mathcal H}}

\def\cN{{\mathcal N}}

\def\cX{{\mathcal X}}

\def\a{{\boldsymbol a}}
\def\b{{\boldsymbol b}}

\def\e{{\boldsymbol e}}
\def\g{{\boldsymbol g}}
\def\l{{\boldsymbol l}}
\def\m{{\boldsymbol m}}

\def\p{{\boldsymbol p}}
\def\q{{\boldsymbol q}}
\def\s{{\boldsymbol s}}

\def\u{{\boldsymbol u}}
\def\v{{\boldsymbol v}}
\def\w{{\boldsymbol w}}
\def\x{{\boldsymbol x}}
\def\y{{\boldsymbol y}}
\def\z{{\boldsymbol z}}
\def\0{{\boldsymbol 0}}
\def\1{{\boldsymbol 1}}

\def\bbeta{{\boldsymbol{\beta}}}

\def\llambda{{\boldsymbol{\lambda}}}
\def\mmu{{\boldsymbol{\mu}}}
\def\nnu{{\boldsymbol{\nu}}}

\def\supp{{\rm supp}}
\def\nsupp{{\rm nsupp}}

\def\deg{{\rm deg}}

\def\Ker{{\rm Ker}}

\def\ini{{\rm in}}
\def\ind{{\rm ind}}
\def\find{{\rm find}}
\def\rank{{\rm rank}}
\def\vol{{\rm vol}}

\def\NS{{\rm NS}}

\title{Logarithmic $A$-hypergeometric series ${\textrm I}\! {\textrm I}$}
\author{Go Okuyama and Mutsumi Saito}
\keywords{$A$-hypergeometric systems; the method of Frobenius.}
\subjclass[2020]{Primary: 33C70}
\address{
\begin{flushleft}
(Go Okuyama) 
Higher Education Support Center\\
Hokkaido University of Science\\
Sapporo, 006-8585, Japan,\\
(Mutsumi Saito) 
Department of Mathematics, Faculty of Science\\
Hokkaido University\\
Sapporo, 060-0810, Japan\\
\end{flushleft}
}
\email{gokuyama@hus.ac.jp (G.Okuyama), saito@math.sci.hokudai.ac.jp (M.Saito)} 

\begin{document}

\begin{abstract}
  In this paper, following \cite{Log}, we continue to develop the perturbing 
method of constructing logarithmic series solutions 
to a regular $A$-hypergeometric system.

  Fixing a fake exponent of an $A$-hypergeometric system, we consider some spaces of linear partial differential operators with constant coefficients.
  Comparing these spaces, we construct a fundamental system of series solutions with the given exponent by the perturbing method. In addition,
we give a sufficient condition for a given fake exponent to be an exponent.
  As important examples of the main results, we give fundamental systems of series solutions to Aomoto-Gel'fand systems and to Lauricella's $F_C$ systems with special parameter vectors, respectively.
\end{abstract}

\maketitle

\section{Introduction}
Let $A=(\a_1,\ldots, \a_n)=(a_{ij})$ be a $d\times n$-matrix
of rank $d$ with coefficients in $\Z$. 
Throughout this paper,
we assume the homogeneity of $A$, i.e.,
we assume that all $\a_j$ belong to one hyperplane off the origin 
in ${\Q}^d$.
Let ${\N}$ be the set of nonnegative integers.
Let $I_A$ denote the toric ideal in the polynomial ring
${\C}[\partial_\x]={\C}[\partial_1,\ldots,\partial_n]$, i.e.,
\begin{equation}
\label{eqn:ToricIdeal}
I_A=\langle \partial_\x^\u-\partial_\x^\v\,|\, A\u=A\v, \, \u, \v\in {\N}^n
\rangle \subseteq {\C}[\partial_\x].
\end{equation}
Here and hereafter we use the multi-index notation;
for example, $\partial_\x^\u$ means $\partial_1^{u_1}\cdots \partial_n^{u_n}$
for $\u=(u_1,\ldots,u_n)^T$.
Given a column vector $\bbeta=(\beta_1,\ldots,\beta_d)^T\in {\C}^d$,
let $H_A(\bbeta)$ denote the left ideal of the Weyl algebra
\begin{equation}
D={\C}\langle\x,\partial_\x\rangle={\C}\langle x_1,\ldots, x_n,\partial_1,\ldots,
\partial_n\rangle
\end{equation}
generated by $I_A$ and 
\begin{equation}
\label{eqn:EulerOperators}
\sum_{j=1}^n a_{ij}\theta_{j} -\beta_i\qquad
(i=1,\ldots, d),
\end{equation}
where $\theta_j =x_j\partial_j$.
The quotient $M_A(\bbeta)=D/H_A(\bbeta)$ is called
the {\it $A$-hypergeometric system with parameter $\bbeta$},
and a formal series annihilated by $H_A(\bbeta)$
an {\it $A$-hypergeometric series with parameter $\bbeta$}.
The homogeneity of $A$ is known to be equivalent to
the regularity of $M_A(\bbeta)$ by Hotta \cite{Hotta} and
Schulze, Walther \cite{SW}.

Logarithm-free series solutions to $M_A(\bbeta)$ 
were constructed by Gel'fand et al. \cite{GGZ,GZK2}
for a generic parameter $\bbeta$, and
more generally in \cite{SST}.

Note that the logarithmic coefficients of $A$-hypergeometric series solutions are polynomials of $\log x^\b$ $(\b\in L)$ \cite[Proposition 5.2]{LogFree},
where
\begin{equation}
L:=\Ker_\Z(A)=
\{ \u\in \Z^n\,|\, A\u=\0\}.
\end{equation}
To construct logarithmic series solutions,
the second author \cite{Log} introduced a method of perturbation 
by a finite subset $B=\{\b^{(1)},\ldots,\b^{(h)}\}\subset L$, and
explicitly described logarithmic series solutions for a fake exponent and a set $B$ that satisfy certain conditions \cite[Theorems 5.4, 6.2 and Remarks 5.6, 6.3]{Log}.

In this paper, following \cite{Log}, we continue to develop the perturbing 
method of constructing logarithmic series solutions 
to a regular $A$-hypergeometric system.

Fixing a fake exponent of an $A$-hypergeometric system, we consider some spaces of linear partial differential operators with constant coefficients.
Comparing these spaces, we construct a fundamental system of series solutions with the given exponent by the perturbing method. 
In addition,
we give a sufficient condition for a given fake exponent to be an exponent.
As important examples of the main results, we give fundamental systems of series solutions to Aomoto-Gel'fand systems and to Lauricella's $F_C$ systems with special parameter vectors, respectively.

This paper is organized as follows. 
In Section 2, we first recall a power series to perturb from \cite{Log}, 
associated with a fake exponent $\v$ and a linearly independent subset $B$ of 
$L$. 
In particular, we discuss properties of each term $a_\u(\s)$ appearing in the series (for the definition of $a_\u(\s)$, see \eqref{eqn:termtoperturb}), and modify the series by changing the range of the sum from $\NS_\w(\v)$ in \cite{Log} to $\cN$ which incorporates $B$.
We give a refinement of \cite[Theorem 6.2]{Log} as Theorem \ref{thm:refinement}.

In Section 3, for a fake exponent $\v$ of the $A$-hypergeometric ideal $H_A(\bbeta)$ with respect to a generic weight vector $\w$, we recall the structure of the ideal $Q_\v$ associated with the fake indicial ideal $\find_\w(H_A(\bbeta))$ and that of its orthogonal complement $Q_\v^\perp$ defined in \cite[Sections 2.3 and 3.6]{SST}.
We introduce ideals $P_\cN$ and $P_B$ of $\C[\s]$, and their orthogonal complements $P_\cN^\perp$ and $P_B^\perp$.
Then we discuss relations among these ideals.
Under a certain condition, we can derive $Q_\v^\perp$ as the image of a linear map from $P_\cN^\perp$ (Proposition \ref{prop:Key2}, Theorem \ref{thm:Pperp Qperp}).

In Section 4, we give a sufficient condition for a fake exponent $\v$ to be an exponent (Proposition \ref{prop:exponent}).
Then we construct a fundamental system of solutions with the exponent $\v$ (Theorem \ref{thm:PBKandQperp})
by applying Theorem \ref{thm:refinement} and the results of Section 3 under the condition that $B$ is a basis of $L$,
which is the main theorem of this paper.

In Sections 5 and 6, we deal with the Aomoto-Gel'fand systems and Lauricella's $F_C$ systems, which are important examples of $H_A(\bbeta)$.
We discuss a fundamental system of solutions to $H_A(\0)$ in each system. 
In each case, we have a unique fake exponent $\v=\0$. 
Taking a basis $B$ of $L$, we can obtain a fundamental system of series solutions for $\bbeta=\0$.

\section{Refinement of \cite[Theorem 6.2]{Log}}

In this section,
we refine \cite[Theorem 6.2]{Log}.

Recall that for $\v=(v_1,\ldots,v_n)^T\in \C^n$ its support $\supp(\v)$ and 
its {\it negative support} $\nsupp(\v)$
are defined as
\begin{align}
  \supp(\v) &:=\{ j\in \{ 1,\ldots,n\} \,|\, v_j\ne 0\},\\
  \nsupp(\v)&:=\{ j\in \{ 1,\ldots,n\} \,|\, v_j\in \Z_{<0}\},
\end{align}
respectively.

For $\v\in \C^n$ and $\u\in \N^n$, set
\begin{equation}
[\v]_{\u}:= \prod_{j=1}^n v_j(v_j-1)\cdots(v_j-u_j+1).
\end{equation}
Here recall that $\N=\{ 0,1,2,\cdots\}$.

Note that we can uniquely write $\u\in \Z^n$ as the sum
$\u=\u_+-\u_-$ with $\u_+,\u_-\in \N^n$ and
$\supp(\u_+)\cap \supp(\u_-)=\emptyset$.

Let $B=\{\b^{(1)},\ldots, \b^{(h)}\}\subset L$. 
We write the same symbol $B$ for the $n\times h$ matrix $(\b^{(1)},\ldots, \b^{(h)})$.

Set
\begin{equation}
\supp(B):=\bigcup_{k=1}^h\supp(\b^{(k)})\subset \{1,\ldots,n\},
\end{equation}
which means the set of all labels for nonzero rows in $B$.

Let $\s=(s_1,\ldots,s_h)^T$ be indeterminates, and  
let
\[
  (B\s)_j:=\sum_{k=1}^h b_{j}^{(k)}s_k\in \C[\s]:=\C[s_1,\ldots,s_h]
\]
for $j=1,\ldots,n$. 
Set
\begin{equation}
  (B\s)^J:=\prod_{j\in J}(B\s)_j\in \C[\s]
\end{equation}
for $J\subset \{1,\ldots,n\}$.
Note that $(B\s)_j=0$ if $j\notin \supp(B)$, hence we have $(B\s)^J=0$ if $J\not\subset \supp(B)$.
\begin{lemma}
\label{lem:denominonzero-2}
Let $B=\{ \b^{(1)},\ldots, \b^{(h)}\}\subset L$, $\u, \u'\in L$ and $\v\in{\C}^{n}$. 
Let $\s=(s_1,\ldots,s_h)^T$ be indeterminates.
Then $[\v+B\s+\u]_{\u'_+}\neq 0$ if and only if $\nsupp(\v+\u-\u')\subset \supp(B)\cup\nsupp(\v+\u)$.
In particular, $[\v+B\s+\u]_{\u_+}\neq 0$ 
if and only if
$\nsupp(\v)\subset \supp(B)\cup \nsupp(\v+\u)$.
\end{lemma}

\begin{proof}
Note that 
\begin{align}
  &[\v+B\s+\u]_{\u'_+}\\
  &\quad= \prod_{j; u'_j>0}(v_j+(B\s)_j+u_j)\cdots(v_j+(B\s)_j+u_j-u'_j+1).
\end{align}
Hence, $[\v+B\s+\u]_{\u'_+}= 0$ if and only if
there exists $j$ such that
$v_j+u_j-u'_j\in \Z_{<0}$, $v_j+u_j\in\N$, and $b^{(k)}_j=0$ for all $k$.

Hence $[\v+B\s+\u]_{\u'_+}= 0$
if and only if
$\nsupp(\v+\u-\u')\not\subset \supp(B)\cup \nsupp(\v+\u)$.
\end{proof}

Let $\w$ be a generic weight.
Recall that $\v$ is called a {\it fake exponent}
of $H_A(\bbeta)$ with respect to $\w$
if
$A\v=\bbeta$ and
$[\v]_{\u_+}=0$ for all
$\u\in L$ with $\u_+\cdot\w>\u_-\cdot\w$,
where
$\u \cdot\w=\sum_{j=1}^n u_j w_j$.

Throughout this paper, fix a generic weight $\w$, a fake exponent $\v$ of $H_{A}(\bbeta)$ with respect to $\w$. 
In addition, we assume the following for a subset $B=\{\b^{(1)},\ldots, \b^{(h)}\}\subset L$.
\begin{assumption}
  \label{ass3}
  A subset $B=\{\b^{(1)},\ldots,\b^{(h)}\}\subset L$ is linearly independent, hence $\rank(B)=h$, and satisfies
\begin{equation}
  \nsupp(\v)\subset \supp(B)\cup \nsupp(\v+\u)
\end{equation}
for any $\u\in L$.
\end{assumption}

\begin{remark}\label{rem:rem_to_ass3}
  \begin{enumerate}
    \item 
    If $B$ satisfies $\nsupp(\v)\subset\supp(B)$, then Assumption \ref{ass3} holds. 
  For example, this condition holds for each of the following cases: 
  \begin{enumerate}
    \item[(i)] $\supp(B)=\{1,\ldots,n\}$,
    \item[(ii)] $\nsupp(\v)=\emptyset$.
  \end{enumerate}
  \item If $B$ is a basis of $L$, then
  $B$ satisfies
  \begin{equation}
    \supp(B)\cup\nsupp(\v+\u)=\supp(B)\cup \nsupp(\v)
  \end{equation}
  for all $\u\in L$.
  Indeed, since $B$ is a basis of $L$, we see that if $j\notin\supp(B)$ then $\u_j=0$ for all $\u\in L$.
  This implies that $\nsupp(\v+\u)\setminus\supp(B)=\nsupp(\v)\setminus\supp(B)$ for all $\u\in L$. 
  \end{enumerate}
\end{remark}

We abbreviate $\nsupp(\v+\u)$ to $I_\u$ for $\u\in L$.
In particular, $I_\0=\nsupp(\v)$.
Assumption \ref{ass3} can be rewritten as
\begin{equation}
  I_\0\subset \supp(B)\cup I_{\u}
\end{equation}
for all $\u\in L$.

For $\u\in L$, let
\begin{equation}\label{eqn:termtoperturb}
  a_\u(\s):=\frac{[\v+B\s]_{\u_-}}{[\v+B\s+\u]_{\u_+}}.
\end{equation}
\begin{lemma}
\label{lem:supp(b)-2}
Let $\u, \u'\in L$. Under Assumption \ref{ass3}, the following hold.
\begin{enumerate}
\item[(i)]\label{lem:supp(b)-2-(i)}
$a_\u(\s)\neq 0$
if and only if
$I_\u\subset \supp(B)\cup I_\0$, if and only if
$\supp(B)\cup I_\u= \supp(B)\cup I_\0$.
\item[(ii)]\label{lem:supp(b)-2-(ii)}
If $I_\u\cup I_{\u-\u'}\not\subset \supp(B)\cup I_\0$, then
$\partial^{\u'_+}(a_\u(\s)x^{\v+B\s+\u})=0$.
\end{enumerate}
\end{lemma}

\begin{proof}
First, the denominator of $a_\u(\s)$ is not zero 
by Lemma \ref{lem:denominonzero-2}.

(i)\ 
We have
$[\v+B\s]_{\u_-}=0$
if and only if
there exists
$j$ such that
$[v_j+\sum_{k=1}^hs_kb^{(k)}_j]_{-u_j}=[v_j]_{-u_j}=0$,
namely
$v_j\in \N,
v_j+u_j\in \Z_{<0}$, and $b_j^{(k)}=0$ for all $k$.
Hence it is equivalent to saying that
there exists
$j$ such that
$j\in I_\u\setminus(\supp(B)\cup I_\0)$,
or
$I_\u\not\subset \supp(B)\cup I_\0$.

By the assumption, the inclusion
$I_\u\subset \supp(B)\cup I_\0$
is equivalent to the equality
$\supp(B)\cup I_\u= \supp(B)\cup I_\0$.

(ii)\ 
Note that
\begin{equation}
\partial^{\u'_+}(a_\u(\s)x^{\v+B\s+\u})
=a_\u(\s)[\v+B\s+\u]_{\u'_+}x^{\v+B\s+\u-\u'_+}.
\end{equation}
Hence, if $I_\u\not\subset \supp(B)\cup I_\0$, then
$\partial^{\u'_+}(a_\u(\s)x^{\v+B\s+\u})=0$ by (\textrm{i}).

Suppose that $I_{\u-\u'}\setminus I_\u \not\subset \supp(B)\cup I_\0$.
Then, by Assumption \ref{ass3}, $I_{\u-\u'}\setminus (\supp(B)\cup I_\u)\ne\emptyset$.
Hence, by Lemma \ref{lem:denominonzero-2}, $[\v+B\s+\u]_{\u'_+}=0$
and we have
$\partial^{\u'_+}(a_\u(\s)x^{\v+B\s+\u})=0$.
\end{proof}

We recall the definitions related to $\NS_\w(\v)$
for a fake exponent $\v$ from \cite{Log} and modify them.

Let
\begin{equation}
  \cG:=\left\{ {\partial^{\g^{(i)}_+}}-\partial^{\g^{(i)}_-}\,\bigl|\, i=1,\ldots, m \right\}
\end{equation}
denote the reduced Gr\"{o}bner basis of $I_A$ with respect to $\w$ with
$\partial^{\g^{(i)}_+}\in \ini_\w (I_A)$ for all $i$. Note that the $\cG$ in \cite[Section 4]{Log} should be {\it the reduced} Gr\"{o}bner basis. 
Set
\begin{equation}
C(\w):=\sum_{i=1}^m \N \g^{(i)}.
\end{equation}

A collection $\NS_\w(\v)$ of negative supports $I_\u$ ($u\in L$) is defined by 
\begin{equation}
\NS_\w(\v):=
\left\{ I_\u\,\middle|\, \u\in L.\  
\text{If $I_{\u}=
I_{\u'}$ for $\u'\in L$, then $\u'\in C(\w)$.}
\right\}.
\end{equation}
In addition, define
\begin{equation}
  \NS_\w(\v)^c:=
  \left\{ I_\u\,\middle|\, \u\in L\right\}\setminus \NS_\w(\v).
\end{equation}
We modify the definition of $\NS_{\w}(\v)$ under Assumption \ref{ass3}.
Define 
\begin{equation}
  \label{eqn:rangeofterms}
  \cN:=\{I\in\NS_\w(\v)\,|\,\supp(B)\cup I=\supp(B)\cup I_\0\},
\end{equation}
and set 
\begin{equation}
  \cN^c:=\{I_\u\,|\,\u\in L, \supp(B)\cup I_\u=\supp(B)\cup I_\0\}\setminus \cN.
\end{equation}
Consider the subset $L'$ of $L$ defined by  
\begin{equation}
\label{eqn:L'}
  L':=\{\u\in L\,|\,I_{\u}\in \cN\}.
\end{equation}
By definition, we see that $L'\subset C(\w)$.

Let
\begin{equation}
  K_{\cN}:=\bigcap_{I\in \cN}I,
\end{equation}
and define the homogeneous ideal $P_\cN$ of $\C[\s]$ for $\cN$ as
\begin{equation}\label{eqn:P_cN}
  P_\cN:=
\left\langle
(B\s)^{I\cup J\setminus K_\cN}\,\biggl|\,I\in \cN,
J\in \cN^c \right\rangle.
\end{equation}
In addition, we define the orthogonal complement $P^\perp$ for a homogeneous ideal $P\subset \C[\s]$ as
\begin{align}
  \label{eqn:defofPprep}
P^\perp&:=\{q(\partial_\s)\in \C[\partial_\s]\,|\,(q(\partial_\s)\bullet h(\s))|_{\s=\0}=0\ \text{for all}\ h(\s)\in P\}\\
&=\{q(\partial_\s)\in \C[\partial_\s]\,|\,q(\partial_\s)\bullet P\subset \langle s_1,\ldots, s_h\rangle\},
\end{align}
where $\C[\partial_\s]:=\C[\partial_{s_1},\ldots,\partial_{s_h}]$.
Since $P$ and $\langle s_1,\ldots, s_h\rangle$ are both homogeneous, $P^\perp$ is homogeneous with respect to the usual total ordering.

\begin{example}(cf.\,\cite[Examples 3.3, 4.8, 6.4]{Log})
  \label{ex:SST-ex3.5.2_(1)}
  Let $n=5$, $d=3$, and 
  \begin{equation}
    A=\begin{bmatrix} 1 & 1 & 1 & 1 & 1 \\ -1 & 1 & 1 & -1 & 0 \\ -1 & -1 & 1 & 1 & 0 \end{bmatrix}.
  \end{equation}
  Let $\bbeta=(1,0,0)^T$ and $\w=(1, 1, 1, 1, 0)$.
  Then $\v=(0,0,0,0,1)^T$ is a unique exponent, and
  \begin{equation}
    \cG=\{\underline{\partial_{x_1}\partial_{x_3}}-\partial_{x_5}^2,\underline{\partial_{x_2}\partial_{x_4}}-\partial_{x_5}^2\},
  \end{equation}
  where the underlined terms are the leading ones.
  Put $\g^{(1)}:=(1,0,1,0,-2)^T$ and $\g^{(2)}:=(0,1,0,1,-2)^T$.
  Recall that
  \begin{align}
    \NS_\w(\v)&=\{\emptyset=I_\0, \{5\}\},\\
    \NS_\w(\v)^c&=\{\{1,3\},\{2,4\},\{1,3,5\},\{2,4,5\},\{1,2,3,4\}\}.
  \end{align}
  Let $B:=\{\g^{(1)}, \g^{(2)}\}$. 
  Then we have $\supp(B)=\{1,2,3,4,5\}$, and
  \begin{align}
    \cN&=\NS_\w(\v),\\ 
    \cN^c&=\NS_\w(\v)^c,\\ 
    K_\cN&=\emptyset.
  \end{align}
  The homogeneous ideal $P_\cN\subset \C[\s]=\C[s_1,s_2]$ and the vector space $P_\cN^\perp\subset \C[\partial_\s]=\C[\partial_{s_1},\partial_{s_2}]$ are given as
  \begin{align}
    P_\cN&=\langle (B\s)^{\{1,3\}},(B\s)^{\{2,4\}}\rangle=\langle s_1^2, s_2^2\rangle,\\ 
    P_\cN^\perp &= \{q(\partial_{s_1},\partial_{s_2})\in \C[\partial_{s_1},\partial_{s_2}]\,|\, q(\partial_{s_1},\partial_{s_2})\bullet \langle s_1^2, s_2^2\rangle \subset \langle s_1, s_2\rangle\}\\
    &=\C 1+\C \partial_{s_1}+\C \partial_{s_2}+\C\partial_{s_1}\partial_{s_2}.
  \end{align}
  We consider another case.
  Let $B_1=\{\g^{(1)}\}$. 
  Then we have $\supp(B_1)=\{1,3,5\}$ and
  \begin{align}
    \cN_1&=\NS_\w(\v)=\{\emptyset=I_\0, \{5\}\},\\ 
    \cN_1^c&=\{\{1,3\},\{1,3,5\}\},\\ 
    K_{\cN_1}&=\emptyset.
  \end{align}
  The homogeneous ideal $P_{\cN_1}\subset \C[s]$ and the vector space $P_{\cN_1}^\perp\subset \C[\partial_s]$ are given as
  \begin{align}
    P_{\cN_1} &=\langle (Bs)^{\{1,3\}}\rangle=\langle s^2\rangle,\\ 
    P_{\cN_1}^\perp &= \{q(\partial_s)\in \C[\partial_s]\,|\, q(\partial_s)\bullet \langle s^2\rangle \subset \langle s\rangle\}=\C 1+\C \partial_s.
  \end{align}
\end{example}

\vspace{24pt}
Throughout this paper, put
\begin{equation}
  m(\s):=(B\s)^{I_\0\setminus K_\cN}.
\end{equation}
The following lemma guarantees that we may plug $\s=\0$ into the series appearing in Theorem \ref{thm:refinement}.

\begin{lemma}\label{lem:expansion}
  Let $\cN$ be the one defined by \eqref{eqn:rangeofterms}, and let $\u, \u'\in L$.
  Then, under Assumption \ref{ass3}, each term of the power series  for $m(\s)\cdot a_\u(\s)\cdot [\v+B\s+\u]_{\u'_+}$ in the indeterminates $\s$ is divided by $(B\s)^{I_\u\cup I_{\u-\u'}\setminus K_\cN}$.  
\end{lemma}
\begin{proof}
  By \cite[Lemma 6.1]{Log},
  there exists a formal power series $g(\y)$ in the indeterminates $\y=(y_1,\ldots,y_n)$ such that
  \begin{align}
    a_\u(\s)&\cdot [\v+B\s+\u]_{\u'_+}\\
    &=\frac{[\v+B\s]_{\u_-}}{[\v+B\s+\u]_{\u_+}}\cdot [\v+B\s+\u]_{\u'_+}\\
    &=\left(\frac{(B\s)^{I_\u\setminus I_\0}}{(B\s)^{I_\0\setminus I_\u}}\cdot (B\s)^{I_{\u-\u'}\setminus I_\u}\right)\cdot g((B\s)_1,\ldots,(B\s)_n)\\
    &=\frac{(B\s)^{(I_\u\cup I_{\u-\u'})\setminus I_\0}}{(B\s)^{I_\0\setminus (I_\u\cup I_{\u-\u'})}}\cdot g((B\s)_1,\ldots,(B\s)_n).
  \end{align}
Hence we have 
\begin{align}
  m(\s)&\cdot a_\u(\s)\cdot [\v+B\s+\u]_{\u'_+}\\
  &=(B\s)^{I_\0\setminus K_\cN}\cdot \frac{(B\s)^{(I_\u\cup I_{\u-\u'})\setminus I_\0}}{(B\s)^{I_\0\setminus (I_\u\cup I_{\u-\u'})}}\cdot g((B\s)_1,\ldots,(B\s)_n)\\
  &=(B\s)^{I_\u\cup I_{\u-\u'}\setminus K_\cN}\cdot g((B\s)_1,\ldots,(B\s)_n),
\end{align}
and the assertion holds.
\end{proof}

We can refine the main results \cite[Theorem 5.4, Theorem 6.2]{Log} as follows.
\begin{theorem}
\label{thm:refinement}
Let $\cN$ be the one defined by \eqref{eqn:rangeofterms}.
Set
\begin{equation}\label{eqn:seriestoperturb}
F_{\cN}(\x,\s)
:=\sum_{\u\in L'}
a_\u(\s)x^{\v+B\s+\u},
\end{equation}
and
\begin{equation}
\widetilde{F}_\cN(\x,\s)
:=m(\s)F_\cN(\x,\s),
\end{equation}
where $L'$ is defined by \eqref{eqn:L'}.

Then $(q(\partial_{s})\bullet
\widetilde{F}_\cN(\x,\s))_{|\s=\0}$
are solutions to $M_A(\bbeta)$ for any $q(\partial_\s)\in P_\cN^\perp$.
\end{theorem}

\begin{proof}
Let $\u'\in L$ and $\u\in L'$.
If $\partial^{\u'_+}(a_\u(\s)x^{\v+B\s+\u})\ne 0$, then we have $I_\u\cup I_{\u-\u'}\subset \supp(B)\cup I_\0$ by Lemma \ref{lem:supp(b)-2-(ii)}, and hence $\supp(B)\cup I_{\u-\u'}=\supp(B)\cup I_\0$ by Assumption \ref{ass3}.
Thus $I_{\u-\u'}\notin \cN$ implies $I_{\u-\u'}\in \cN^c$. 

Similar to the arguments in the proofs of \cite[Theorem 5.4, Theorem 6.2]{Log}, we see that
\begin{align}
  (\partial^{\u'_+}&-\partial^{\u'_-})\bullet \widetilde{F}_\cN(\x,\s)\\
  &=\sum_{\u\in L', I_{\u-\u'}\in \cN^c}m(\s)\partial^{{\u'}_+}(a_\u(\s)x^{\v+B\s+\u})\\
  &\qquad -\sum_{\u\in L', I_{\u+\u'}\in\cN^c}m(\s)\partial^{{(-\u)'}_+}(a_\u(\s)x^{\v+B\s+\u}).
\end{align}
Let $q(\partial_\s)\in P_\cN^\perp$. 
Then the series $(q(\partial_{s})\bullet
\widetilde{F}_\cN(\x,\s))_{|\s=\0}$
is a solution to $M_A(\bbeta)$ if 
\begin{equation}
  \label{eqn:termIandJ}
  \left(q(\partial_{s})\bullet
  \left(m(\s)
  \partial^{\u'_+}a_\u(\s)x^{\v+B\s+\u}\right)\right)_{|\s=\0}=0
\end{equation}
for any $\u\in L'$ and $\u'\in L$ with $I_{\u-\u'}\in \cN^c$.


By Lemma \ref{lem:expansion}, each coefficient of
\begin{align}
  \label{eqn:IJlowest}
  m(\s)&(\partial^{\u'_+}a_\u(\s)x^{\v+B\s+\u})\\
  &=m(\s)a_\u(\s)[\v+B\s+\u]_{\u'_+}x^{\v+B\s+\u-\u'_+}
\end{align}
in the indeterminates $\s$ is divided by $(B\s)^{I_\u\cup I_{\u-\u'}\setminus K_\cN}$, hence belongs to $P_\cN$. 
By the definition of $P_\cN^\perp$, the assertion holds.
\end{proof}


\section{Relations between $P_\cN^\perp$ and $Q_\v^\perp$}

In this section, we recall $Q_\v$ and its orthogonal complement $Q_\v^\perp$ defined in \cite[Section 2.3]{SST}, and discuss relations between $P_\cN^\perp$ and $Q_\v^\perp$. For the definitions of $P_\cN$ and $P_\cN^\perp$, see \eqref{eqn:P_cN} and \eqref{eqn:defofPprep}. 

Consider the fake indicial ideal $\find_{\w}(H_A(\bbeta))$ of $H_A(\bbeta)$ with respect to $\w$:
\begin{equation}\label{eqn:fake_ind}
  \find_{\w}(H_A(\bbeta)):=\langle A\theta_\x-\bbeta\rangle+\widetilde{\ini}_\w(I_A)\subset \C[\theta_\x]:=\C[\theta_{1},\ldots,\theta_{n}].
\end{equation}
Here $\widetilde{\ini}_\w(I_A)$ is the distraction of the initial ideal $\ini_\w(I_A)$ with respect to $\w$ (cf.\,\cite[Section 3.1]{SST}).
Related to the reduced Gr\"{o}bner basis $\cG=\{\g^{(i)}\,|\,i=1,\ldots,m\}$ of $I_A$ with respect to $\w$, define
\begin{equation}
G^{(i)}:=I_{-\g^{(i)}}\setminus I_\0=\{j\in \{1,\ldots,n\}\,|\, v_j\in\N, g_j^{(i)}-v_j>0\}
\end{equation}
for $i=1,\ldots,m$.
Since 
\begin{equation}
\widetilde{\ini}_\w (I_A)
=\left\langle\,
[\theta_\x]_{\g^{(i)}_+}:=\prod_{j;\,g_j^{(i)}>0}\prod_{\nu=0}^{g_j^{(i)}-1}(\theta_j-\nu)\,\biggl|\,
i=1,\ldots,m\,
\right\rangle
\end{equation}
by \cite[Theorem 3.2.2]{SST}, 
we see that its primary component at a fake exponent $\v$ is 
\begin{equation}
\widetilde{\ini}_\w(I_A)_\v
=
\left\langle\,
([\theta]_{\g^{(i)}_+})_\v:=
\prod_{j\in G^{(i)}}(\theta_j-v_j)\,\biggl|\,
i=1,\ldots,m\,
\right\rangle.\label{ass1}
\end{equation}

We obtain the homogeneous ideal $Q_\v$ of $\C[\theta_\x]$ from $\find_\w(H_A(\bbeta))_\v$
by replacing $\theta_j\mapsto \theta_j+v_j$ for $j=1,\ldots,n$ (cf.\,\cite[Section 2.3]{SST}).
Namely, 
\begin{equation}
  \label{eqn:Q_v}
  Q_\v=\left\langle\,A\theta_\x\,\right\rangle+\left\langle\prod_{j\in G^{(i)}}\theta_j\,\biggl|\,
  i=1,\ldots,m\,\right\rangle.
\end{equation}
The orthogonal complement $Q_\v^\perp$ of $Q_\v$ is defined by
\begin{equation}
Q_\v^\perp:=
\{
f\in \C[\x]\,|\, \text{$\phi(\partial_\x)(f)=0$
for all $\phi=\phi(\theta_\x)\in Q_\v$}
\}.
\end{equation}
Note that $Q_\v^\perp$ is a graded $\C$-vector space with the usual grading.
\begin{proposition}
  \label{prop3}
  Let $f(\x)$ be a polynomial.
  Then $x^\v f(\log\x)$ is a solution to 
  $\find_\w(H_A(\bbeta))$
   if and only if $f(\x)$ satisfies the following conditions:
  \begin{enumerate}
  \item[(i)]
  $f(\x)\in \C[\x G]:=\C[\x \g^{(1)},\ldots,\x \g^{(m)}]$.
  \item[(ii)]
  $\partial_\x^{G^{(i)}}\bullet f(\x)=0$ for all $i=1,\ldots,m$.
  \end{enumerate}
  Here
  \begin{equation}
    \x G:=(\x \g^{(1)},\ldots,\x \g^{(m)})=\left(\sum_{j=1}^n g^{(1)}_jx_j,\ldots,\sum_{j=1}^n g^{(m)}_jx_j\right)
  \end{equation}
  for $\x=(x_1,\ldots,x_n)$.
  \end{proposition}
  
  \begin{proof}
  By \cite[Theorem 2.3.11]{SST},
  the function $x^\v f(\log\x)$ is a solution to $\find_\w(H_A(\beta))$ if and only if $f(\x)\in Q_\v^\perp$.
  From $f(\x)\in \langle A\partial_\x\rangle^\perp$,
  we see (i) \cite[Lemma 5.1]{LogFree}.
  (ii) follows from Equation \eqref{ass1}.
  \end{proof}

\begin{example}[Continuation of Example \ref{ex:SST-ex3.5.2_(1)}]
  \label{ex:SST-ex3.5.2_(2)}
  Note that $\v-\g^{(1)}=(-1,0,-1,0,3)^T$ and $\v-\g^{(2)}=(0,-1,0,-1,3)^T$.
  Thus we see that 
  \begin{align}
    G^{(1)}&=I_{-\g^{(1)}}\setminus I_\0=\nsupp(\v-\g^{(1)})\setminus \nsupp(\v)=\{1,3\},\\ 
    G^{(2)}&=I_{-\g^{(2)}}\setminus I_\0=\nsupp(\v-\g^{(1)})\setminus \nsupp(\v)=\{2,4\}.
  \end{align}
  The ideal $Q_\v\subset \C[\theta_{\x}]=\C[\theta_1, \theta_2, \theta_3, \theta_4, \theta_5]$ is given as
  \begin{equation}
    Q_\v=\langle \theta_1+\theta_2+\theta_3+\theta_4+\theta_5, -\theta_1+\theta_2+\theta_3-\theta_4,-\theta_1-\theta_2+\theta_3+\theta_4, \theta_1\theta_3, \theta_2\theta_4\rangle.
  \end{equation}
  In addition, we see that 
  \begin{align}
    Q_\v^\perp&=\C\cdot 1+\C\cdot \x\g^{(1)}+\C\cdot \x\g^{(2)}+\C\cdot (\x\g^{(1)})\cdot (\x\g^{(2)}).
  \end{align}
\end{example}
To compare $Q_\v$ with $P_\cN$, we consider the graded ring homomorphism $\Phi_B: \C[\theta_\x]\rightarrow \C[\s]$ defined by $\theta_j\mapsto (B\s)_j$ for $j=1,\ldots,n$.
By the linear independence of $B$, we see that $\Phi_B$ is surjective.
Define $P_{B}:=\Phi_B(Q_\v)$.
By the ring isomorphism theorem, $\Phi_B$ induces the ring isomorphism
\begin{equation}
  \widetilde{\Phi}_B: \C[\theta_\x]/\Phi_B^{-1}(P_B)\simeq \C[\s]/P_B.
\end{equation}
Since $\langle A\theta_\x\rangle$ is vanished by $\Phi_B$, we have
\begin{equation}\label{eqn:P_B_def}
  P_B=\left\langle(B\s)^{G^{(i)}}\,\Bigl|\, i=1,\ldots, m\right\rangle.
\end{equation}

  
\begin{proposition}
\label{prop:Assump(2)}
Let $J\in {\NS_\w(\v)}^c$.
Then
$G^{(i)}\subset J\setminus I_\0$ for some $i$.
\end{proposition}
  
\begin{proof}
By definition and \cite[Lemma 4.2]{Log}, we see that there exists $\u\in L\setminus C(\w)$ such that $J=I_\u$ and $\partial^{\u_+}\notin \ini_\w(I_A)$.
Hence $\partial^{\u_-}=\ini_\w(\partial^{\u_-}-\partial^{\u_+})$ is divided by some $\partial^{\g^{(i)}_+}$.
Let $j\in G^{(i)}=I_{-\g_+^{(i)}}\setminus I_\0$. 
Then $v_{j}\in \N$ and $v_j-g_j^{(i)}\in\Z_{<0}$.
  Since $g_j^{(i)}\in\Z_{>0}$, we see that $g_j^{(i)}\leq -u_j$ and $v_{j}+u_j\leq v_j-g^{(i)}_j<0$.
Thus we have $j\in I_\u\setminus I_\0=J\setminus I_\0$.
\end{proof}
Three ideals $Q_\v$, $P_{\cN}$, and $P_B$ are related as follows.
\begin{proposition}\label{prop:Key1}
  Let $Q_\v$, $P_{\cN}$, and $P_B$ be the ones in \eqref{eqn:Q_v}, \eqref{eqn:P_cN}, and \eqref{eqn:P_B_def}, respectively.
  Then, the following hold.
  \begin{enumerate}
    \item[(i)] 
    $m(\s)\cdot P_B\subset P_\cN\subset P_B$. 
    In particular, if $K_{\cN}=I_0$, then $P_{\cN}=P_B$. 
    \item[(ii)] If $B$ is a basis of $L$, then $\Phi_B^{-1}(P_B)=Q_\v$.
  \end{enumerate}
\end{proposition}
\begin{proof}
  (i)\ Let $I\in\cN$ and $J\in{\cN}^c$.
  Since $J\in {\NS_\w(\v)}^c$ and $K_{\cN}\subset I_\0$,  $I\cup J\setminus K_{\cN}$ contains some $G^{(i)}$ by Proposition \ref{prop:Assump(2)}.
  Hence the inclusion $P_{\cN}\subset P_B$ holds.

  For any $i=1,\ldots,m$, since $-\g^{(i)}\notin C(\w)$ we see that $I_{-\g^{(i)}}\in{\NS_\w(\v)}^c$.
  If $I_{-\g^{(i)}}\notin \cN^c$, then 
  \begin{equation}\label{eqn:inequality_for_supp(B)}
    \supp(B)\cup I_{-\g^{(i)}}\ne\supp(B)\cup I_\0.
  \end{equation}
  By Assumption \ref{ass3}, \eqref{eqn:inequality_for_supp(B)} implies that $G^{(i)}=I_{-\g^{(i)}}\setminus I_\0\not\subset \supp(B)$.
  Hence we have $(B\s)^{G^{(i)}}=0$.
  If $I_{-\g^{(i)}}\in \cN^c$, then 
  \begin{equation}
    m(\s)(B\s)^{G^{(i)}}=(B\s)^{I_\0\cup I_{-\g^{(i)}}\setminus K_\cN}\in P_\cN.
  \end{equation}
  Hence we have $m(\s)\cdot P_B\subset P_\cN$.

  (ii)\ Since $B$ is a basis of $L$, we have $\Ker(\Phi_B)=\langle A\theta_\x\rangle$.
  Thus the assertion ${\Phi_B}^{-1}(P_B)= Q_{\v}$ holds from \eqref{eqn:Q_v}.
\end{proof}
\begin{example}[Continuation of Example \ref{ex:SST-ex3.5.2_(1)} and \ref{ex:SST-ex3.5.2_(2)}]
  \label{ex:SST-ex3.5.2_(3)}
  Consider the case where $B=\{\g^{(1)}, \g^{(2)}\}$.
  Then we have $m(\s)=(B\s)^{\emptyset}=1$, and
  \begin{equation}
    P_B=\langle (B\s)^{G^{(1)}}, (B\s)^{G^{(2)}}\rangle =\langle s_1^2, s_2^2\rangle=P_\cN.
  \end{equation}
  Furthermore, since $B$ is a basis of $L$, we see that 
  \begin{equation}
    \Phi_B^{-1}(P_B)=\Phi_B^{-1}(\langle (B\s)^{G^{(1)}}, (B\s)^{G^{(2)}}\rangle)=\langle \theta_1\theta_3, \theta_2\theta_4\rangle+\langle A\theta_\x\rangle=Q_\v.
  \end{equation}
  Consider the other case where $B_1=\{\g^{(1)}\}$.
  Then we have $m(s)=(B_1 s)^{\emptyset}=1$, and
  \begin{equation}
    P_{B_1}=\langle (B_1 s)^{G^{(1)}}\rangle =\langle s^2 \rangle=P_{\cN_1}.
  \end{equation}
  We see that $B_1$ does not span $L$ and that   
  \begin{equation}
    \Phi_B^{-1}(P_{B_1})=\Phi_{B_1}^{-1}(\langle (B_1 s)^{G^{(1)}}\rangle)=\langle \theta_1\theta_3\rangle+\langle A\theta_\x\rangle\subsetneq Q_\v.
  \end{equation}
\end{example}

\vspace{24pt}
We consider relations between $P_\cN^\perp$ and $P_B^\perp$, and between $P_B^\perp$ and $Q_\v^\perp$.
Recall the construction of a basis of orthogonal complements in \cite[Section 2.3]{SST}.

Let $P$ be a homogeneous ideal of $\C[\s]$.
Fix any term order $\prec$ on $\C[\s]$, and let $\cH\subset \C[\s]$ be the reduced Gr\"{o}bner basis of $P$ with respect to $\prec$.
For any $\mmu\in \N^h$ with $\s^\mmu\in\ini_\prec(P)$, there exist unique $c_{\mmu, \nnu}\in \C$ for $\nnu\in \N^h$ with $|\nnu|=|\mmu|$ and $\s^\nnu\notin\ini_\prec(P)$ such that 
\begin{equation}
  p_\mmu(\s):=\s^\mmu-\sum_{\substack{\nnu\in\N^h;\,|\nnu|=|\mmu|,\\\s^\nnu\notin\ini_\prec(P)}}c_{\mmu,\nnu}\s^{\nnu}\in P.
\end{equation}
We obtain $p_\mmu(\s)$ by taking the normal form  modulo $\cH$ for the monomial $\s^\mmu$.
For $\nnu\in\N^h$ with $\s^\nnu\not\in\ini_\prec(P)$, define the homogeneous polynomial $q_\nnu(\partial_\s)$ of degree $|\nnu|$ by
\begin{equation}
  q_\nnu(\partial_\s):=\frac{1}{\nnu!}\partial_\s^\nnu +\sum_{\substack{\mmu\in\N^h;\,|\mmu|=|\nnu|,\\\s^\mmu\in \ini_\prec(P)}}\frac{c_{\mmu,\nnu}}{\mmu!}\partial_\s^\mmu\in \C[\partial_\s].
\end{equation}
  
\begin{lemma}\label{lem:Basis_of_Pperp}
  Let $P$ be a homogeneous ideal of $\C[\s]$.
  Fix any term order $\prec$ on $\C[\s]$, and let $\cH\subset \C[\s]$ be the reduced Gr\"{o}bner basis of $P$ with respect to $\prec$.
  Then 
  \begin{equation}
    \{p_\mmu(\s)\,|\,\s^\mmu\in\ini_\prec(P)\}
  \end{equation}
  and 
  \begin{equation}
    \{q_\nnu(\partial_\s)\,|\,\nnu\in\N^h\ \text{with}\ \s^\nnu\notin \ini_\prec(P)\}
  \end{equation}
  form $\C$-bases of $P$ and $P^\perp$, respectively.
\end{lemma}
\begin{proof}
  This is similar to \cite[Proposition 2.3.13]{SST}.
\end{proof}
\begin{lemma}\label{lem:duality}
  Let $P$ and $\widetilde{P}$ be homogeneous ideals of $\C[\s]$.
  Then $P\subset \widetilde{P}$ if and only if $P^\perp\supset {\widetilde{P}}^\perp$.
\end{lemma}
\begin{proof}
  Assume that $P\subset \widetilde{P}$. 
  By definition, $P^\perp\supset {\widetilde{P}}^\perp$ is clear.

  Conversely, assume that $P^\perp\supset {\widetilde{P}}^\perp$. 
  Fix any term order $\prec$ on $\C[\s]$, and let $\widetilde{\cH}$ be the reduced Gr\"{o}bner basis of $\widetilde{P}$.

  Let $\{\widetilde{q}_\nnu(\partial_\s)\,|\,\nnu\in\N^h\ \text{with}\ \s^\nnu\notin \ini_\prec(\widetilde{P})\}$ be the $\C$-basis of $\widetilde{P}^\perp$
  as in Lemma \ref{lem:Basis_of_Pperp}.
  Let $p\in P$. 
  Applying the division algorithm  with respect to $\widetilde{\cH}$ to $p$, we can express $p$ as 
  \begin{equation}
    p=\widetilde{p}+\sum_{\llambda; \s^\llambda\notin\ini_\prec(\widetilde{P})}d_\llambda s^\llambda
  \end{equation}
  with some $\widetilde{p}\in \widetilde{P}$ and $d_\llambda\in\C$.
  For each $\nnu\in\N^h$ with $\s^\nnu\not\in \ini_\prec(\widetilde{P})$, since $[\partial_\s^\mmu\bullet \s^\llambda]_{|\s=\0}=\mmu!\delta_{\mmu,\llambda}$ for any $\mmu$ and $\llambda$, we have 
  \begin{align}
    [\widetilde{q}_\nnu(\partial_\s)&\bullet p]_{|\s=\0}\\
    &=\left[\widetilde{q}_\nnu(\partial_\s)\bullet \left(\widetilde{p}+\sum_{\llambda; \s^\llambda\notin\ini_\prec(\widetilde{P})}d_\llambda s^\llambda\right)\right]_{|\s=\0}\\
    &=\sum_{\llambda; \s^\llambda\notin\ini_\prec(\widetilde{P})}d_\llambda\left[\widetilde{q}_\nnu(\partial_\s)\bullet \s^\llambda\right]_{|\s=\0}\\
    &=\sum_{\llambda; \s^\llambda\notin\ini_\prec(\widetilde{P})}d_\llambda\left[\left(\frac{1}{\nnu!}\partial_\s^\nnu +\sum_{\substack{\mmu\in\N^h;\,|\mmu|=|\nnu|,\\\s^\mmu\in \ini_\prec(\widetilde{P})}}\frac{c_{\mmu,\nnu}}{\mmu!}\partial_\s^\mmu\right)\bullet \s^\llambda\right]_{|\s=\0}\\
    &=d_\nnu.
  \end{align}
  It follows from the assumption $P^\perp\supset \widetilde{P}^\perp$ that $d_\nnu=0$ for all $\nnu\in\N^h$ with $\s^\nnu\notin\ini_\prec(\widetilde{P})$.
  Hence, we have $p\in \widetilde{P}$.
\end{proof}

Let $\C[\partial_\z]:=\C[\partial_{z_1},\ldots,\partial_{z_h}]$ be the ring of partial differential operators with constant coefficients in indeterminates $\z=(z_1,\ldots,z_h)$.
To describe relations between $P_\cN^\perp$ and $P_B^\perp$, and between $P_B^\perp$ and $Q_\v^\perp$, we define an action of $\C[\partial_\z]$ on $\C[\partial_{\s}]$ and a ring homomorphism $\Psi_B$ from $\C[\partial_\s]$ to $\C[\x]$.

For $U(\partial_\z)\in\C[\partial_\z]$ and $q(\partial_\s)\in\C[\partial_\s]$, we define a $\C$-linear operation $U(\partial_{z_1},\ldots,\partial_{z_h})\star  q(\partial_\s)$ by 
\begin{equation}
  U(\partial_{z_1},\ldots,\partial_{z_h})\star q(\partial_\s):=(U(\partial_\z)\bullet q(\z))|_{\z=\partial_\s}\in \C[\partial_\s].
\end{equation}

\begin{lemma}\label{lem:staroperation}
  The following hold for the $\star$-operation.
  \begin{enumerate}
    \item[(i)] Let $k=1, \ldots, h$ and $q(\partial_\s)\in \C[\partial_\s]$.
    Then  
    \begin{equation}
      \partial_{z_{k}}\star q(\partial_\s)=q(\partial_\s)s_k-s_k q(\partial_\s)\in \C\langle \s, \partial_\s\rangle.
    \end{equation}
    \item[(ii)] Let $U(\partial_z), U'(\partial_z)\in \C[\partial_\z]$, and $q(\partial_\z)\in \C[\partial_\s]$. 
    Then
    \begin{equation}
      U(\partial_\z)\star (U'(\partial_\z)\star q(\partial_\s))=(U(\partial_\z)U'(\partial_\z))\star q(\partial_s).
    \end{equation}
    \item[(iii)] Let $U(\partial_\z)=\prod_{\nu=1}^{N}l_\nu(\partial_\z) \in \C[\partial_\z]$ be the product of non-zero linear homogeneous polynomials $l_\nu(\partial_\z)$, and let $q(\partial_\s)\in\C[\partial_\s]$.
    Then there exists $r(\partial_\s)\in\C[\partial_\s]$ such that $U(\partial_\z)\star r(\partial_\s)=q(\partial_\s)$.
    \end{enumerate}
\end{lemma}
\begin{proof}
  (i)\ For any $k=1,\ldots,h$, and $\mmu\in\N^h$, we have  
  \begin{equation}
    \partial_{z_k}\star \partial_\s^\mmu=\mu_k \partial_\s^{\mmu-\e_k}=\partial_\s^\mmu s_k -s_k \partial_\s^\mmu.
  \end{equation}
  
  (ii)\ It suffices to show that the equality holds for $U(\partial_\z)=\partial_\z^\llambda$, $U'(\partial_\z)=\partial_\z^\mmu$, and $q(\partial_\s)=\partial_\s^\nnu$ with $\llambda, \mmu, \nnu\in\N^h$. 
  We see that 
  \begin{align}
    \partial_\z^\llambda\star(\partial_\z^\mmu \star \partial_\s^\nnu)&=\partial_\z^\llambda\star ((\partial_\z^\mmu \bullet \z^\nnu)|_{|\z=\partial_\s})\\
    &=[\nnu]_\mmu \partial_\z^\llambda \star \partial_\s^{\nnu-\mmu}\\
    &=[\nnu]_\mmu (\partial_\z^\llambda \bullet \z^{\nnu-\mmu})_{|\z=\partial_\s}\\
    &=[\nnu]_\mmu [\nnu-\mmu]_\llambda \partial_\s^{\nnu-\mmu-\llambda}\\
    &=[\nnu]_{\llambda+\mmu}\partial_\s^{\nnu-(\llambda+\mmu)}\\
    &=(\partial_\z^{\llambda+\mmu}\bullet \partial_\s^{\nnu})_{|\z=\partial_\s}\\
    &=\partial_\z^{\llambda+\mmu}\star \partial_\s^\nnu,
  \end{align}
  and hence the assertion holds.

  (iii)\ We show the statement by induction on $N$. 
  First, let $U(\partial_\z)$ be a non-zero linear homogeneous polynomial, and let $q(\partial_\s)\in\C[\partial_\s]$.
  By changing coordinates, we may assume that $U(\partial_\z)=\partial_{z_1}$.
  Put
  \begin{equation}
    q(\partial_\s)=\sum_{\nnu\in\N^h}d_\nnu \partial_\s^\nnu.
  \end{equation} 
  Then 
  \begin{equation}
    r(\partial_\s)=\sum_{\nnu\in\N^h}\frac{d_\nnu}{\nu_1+1}\partial_\s^{\nnu+\e_1}
  \end{equation}
  satisfies $U(\partial_\z)\star r(\partial_\s)=q(\partial_\s)$.

  Next, fix $N>1$, and let $U(\partial_\z)=\prod_{\nu=1}^{N}l_\nu(\partial_\z)$ such that $l_\nu(\partial_\z)$ are non-zero linear homogeneous polynomials.
  Assume that the assertion holds for any product of non-zero linear homogeneous polynomial of degree less than $N$.
  By the induction hypothesis, there exist
  $r(\partial_\s), \widetilde{r}(\partial_\s)\in \C[\partial_\s]$ such that 
  \begin{equation}
    l_{1}(\partial_\z)\star \widetilde{r}(\partial_\s)=q(\partial_\s),\qquad \left(\prod_{\nu=2}^{N}l_\nu(\partial_\z)\right)\star r(\partial_\s)=\widetilde{r}(\partial_\s).
  \end{equation}
  By (ii), we have 
  \begin{align}
    U(\partial_\z)\star r(\partial_\s)&=\left(l_{1}(\partial_\z)\left(\prod_{\nu=2}^{N}l_\nu(\partial_\z)\right)\right)\star r(\partial_\s)\\
    &=l_{1}(\partial_\z)\star \left(\left(\prod_{\nu=2}^{N}l_\nu(\partial_\z)\right)\star r(\partial_\s)\right)\\
    &=l_{1}(\partial_\z)\star \widetilde{r}(\partial_\s)\\
    &=q(\partial_\s),
  \end{align}
  and hence the assertion holds.
\end{proof}

\begin{lemma}\label{lem:Action}
  Let $U(\partial_z)\in \C[\partial_\z]$, $q(\partial_\s)\in\C[\partial_\s]$, and $f(\s)\in \C[[\s]]$.
  Then 
  \begin{equation}
    \left[q(\partial_\s)\bullet(U(\s)f(\s))\right]_{|\s=\0}=\left[\left(U(\partial_\z)\star q(\partial_\s)\right)\bullet f(\s)\right]_{|\s=\0}.
  \end{equation}
\end{lemma}
\begin{proof}
  We show that the statement holds for any monomial operator $U(\partial_\z)=\partial_\z^\mmu$ by induction on $|\mmu|$.
Firstly, the assertion is clear for $\mmu=\0$.

Secondly, assume that $\mmu=\e_k$ for $k=1,\ldots,k$, hence $U(\partial_\z)=\partial_\z^{\mmu}=\partial_{z_k}$.
Note that $[s_k q(\partial_\s)\bullet f(\s)]_{|\s=\0}=0$ for any $k=1\ldots,h$, $q(\partial_\s)\in\C[\partial_\s]$, and $f(\s)\in\C[[ \s ]]$.
Thus, by Lemma \ref{lem:staroperation}, we see that
\begin{align}
  [q(\partial_\s)&\bullet \left(U(\s)f(\s)\right)]_{|\s=\0}\\
  &=\left[(q(\partial_\s)s_k)\bullet f(\s)\right]_{|\s=\0}\\
  &=\left[(s_k q(\partial_\s)+\partial_{z_k}\star q(\partial_\s))\bullet f(\s)\right]_{|\s=\0}\\
  &=\left[s_k (q(\partial_\s)\bullet f(\s))\right]_{|\s=\0}+\left[(\partial_{z_k}\star q(\partial_\s))\bullet f(\s)\right]_{|\s=\0}\\
  &=\left[(\partial_{z_k}\star q(\partial_\s))\bullet f(\s)\right]_{|\s=\0}.
\end{align}
Hence the assertion holds for $|\mmu|=1$.

Finally, fix $\mmu\in\N^h$ with $|\mmu|>1$.
Let $U(\partial_\z)=\partial_\z^\mmu$, $q(\partial_\s)\in \C[\partial_\s]$, and $f(\s)\in \C[[\s]]$.
Assume that the assertion holds for any $\widetilde{U}(\partial_\z)=\partial_{\z}^{\widetilde{\mmu}}$ with $|\widetilde{\mmu}|<|\mmu|$.
Then there exists $k$ such that $\mmu_k>0$. 
Applying the induction hypothesis to the operators $\partial_{z_k}$ and $\partial_\z^{\mmu-\e_k}$, respectively, we see from Lemma \ref{lem:staroperation} (ii) that
\begin{align}
  [q(\partial_\s)&\bullet \left(U(\s)f(\s)\right)]_{|\s=\0}\\
  &=\left[q(\partial_\s)\bullet \left(s_k \cdot \s^{\mmu-\e_k}f(\s)\right)\right]_{|\s=\0}\\
  &=\left[\left(\partial_{z_k}\star q(\partial_\s)\right)\bullet \left(\s^{\mmu-\e_k}f(\s)\right)\right]_{|\s=\0}\\
  &=\left[\left(\partial_\z^{\mmu-\e_k}\star\left(\partial_{z_k}\star q(\partial_\s)\right)\right)\bullet f(\s)\right]_{|\s=\0}\\
  &=\left[\left(U(\partial_\z)\star q(\partial_\s)\right)\bullet f(\s)\right]_{|\s=\0}.
\end{align}
Hence the assertion holds.
\end{proof}
We define a ring homomorphism $\Psi_B: \C[\partial_\s]\rightarrow \C[\x]$ as
\begin{equation}\label{eqn:Psi}
  \Psi_B(q(\partial_\s))(\x):=q(\x B)=q\left(\sum_{j=1}^n b_j^{(1)}x_j, \ldots, \sum_{j=1}^n b_j^{(h)}x_j\right)
\end{equation}
for $q(\partial_\s)\in \C[\partial_\s]$.
Note that $\Psi_B$ is injective by the linear independence of $B$.
\begin{proposition}\label{prop:startingterm}
  Let $q(\partial_\s)\in \C[\partial_\s]$. 
  Then
  \begin{equation}
    \left[q(\partial_\s)\bullet \left(m(\s)x^{\v+B\s}\right)\right]_{|\s=\0}=x^\v\Psi_B\left(m(\partial_\z)\star q(\partial_\s)\right)(\log \x),
  \end{equation}
  where $\log \x:=(\log x_1, \ldots, \log x_n)$.
\end{proposition}
\begin{proof}
  Note that we can regard $x^{\v+B\s}$ as the formal series $x^\v e^{(\log \x) B\s}$ in $\s$, where 
  \begin{equation}
    (\log \x)B\s:=\sum_{j=1}^{n}\sum_{k=1}^{h}(\log x_{j})b_{j}^{(k)}s_{k}.
  \end{equation}
  Put $r(\partial_\s):=m(\partial_\z)\star q(\partial_\s)$.
  Then, by Lemma \ref{lem:Action}, 
  \begin{align}
    \left[q(\partial_\s)\bullet \left(m(\s)x^{\v+B\s}\right)\right]_{|\s=\0}&=\left[\left(m(\partial_\z)\star q(\partial_\s)\right)\bullet x^{\v+B\s}\right]_{|\s=\0}\\
    &=\left[r(\partial_\s)\bullet x^{\v+B\s}\right]_{|\s=\0}\\
    &=[r((\log \x)B)x^{\v+B\s}]_{|\s=\0}\\
    &=x^\v\Psi_B (r(\partial_\s))(\log \x).
  \end{align}
\end{proof}

\begin{proposition}\label{prop:Key2}
  The following hold.
  \begin{enumerate}
    \item[(i)] $m(\partial_\z)\star P_\cN^\perp\subset P_B^\perp\subset P_\cN^\perp$.
    In particular, if $K_{\cN}=I_\0$, then $P_\cN^\perp = P_B^\perp$. 
    \item[(ii)] $m(\s)\in P_\cN$ if and only if $m(\partial_\z)\star P_\cN^\perp=\{0\}$.
    \item[(iii)] If $P_\cN=m(\s)\cdot P_B$, then $m(\partial_\z)\star P_\cN^\perp= P_B^\perp$.
  \end{enumerate}
\end{proposition}
\begin{proof}
  (i)\ $P_B^\perp \subset P_\cN^\perp$ is clear by Lemma \ref{prop:Key1} (i) and Lemma \ref{lem:duality}.
  Let $q(\partial_\s)\in P_\cN^\perp$. 
  Then, for any $f(\s)\in P_B$, Lemma \ref{lem:Action} shows that
  \begin{equation}
    \left[\left(m(\partial_\z)\star q(\partial_\s)\right)\bullet f(\s)\right]_{|\s=\0}=\left[q(\partial_\s)\bullet \left(m(\s)f(\s)\right)\right]_{|\s=\0}.
  \end{equation}
  It follows from Lemma \ref{prop:Key1} (i) that the right hand side is $0$.
  Hence we have $m(\partial_\z)\star P_\cN^\perp\subset P_B^\perp$.

  (ii)\ Assume that $m(\s)\in P_\cN$.
  Let $q(\partial_\s)\in P_\cN^\perp$. Put $m(\partial_\z)\star q(\partial_s)=\sum_{\nnu}a_\nnu \partial_\s^\nnu$, where $a_\nnu\in\C$.
  Then, by Lemma \ref{lem:Action}, we have
  \begin{align}
    \nnu!a_\nnu &= \left[\left(m(\partial_\z)\star q(\partial_s)\right)\bullet \s^\nnu\right]_{|\s=\0} \\
    &=\left[q(\partial_\s)\bullet \left(m(\s)\s^\nnu\right)\right]_{|\s=\0}=0
  \end{align}
  for any $\nnu\in\N^h$.
  Hence, we have $m(\partial_\z)\star q(\partial_s)=0$.

  Conversely, assume that $m(\partial_\z)\star P_\cN^\perp=\{0\}$. Let $q(\partial_\s)\in P_\cN^\perp$.
  Then, Lemma \ref{lem:Action} shows that 
  \begin{equation}
    \left[q(\partial_\s)\bullet \left(m(\s)f(\s)\right)\right]_{|\s=\0}=\left[\left(m(\partial_\z)\star q(\partial_s)\right)\bullet f(\s)\right]_{|\s=\0}=0
  \end{equation}
  for any $f(\s)\in\C[\s]$.
  Thus we have $q(\partial_\s)\in \langle m(\s)\rangle^\perp$, that is, $P_\cN^\perp \subset \langle m(\s)\rangle^\perp$. 
  By Lemma \ref{lem:duality}, $m(\s)\in P_\cN$.

  (iii)\ In (i), we have seen $m(\partial_\z)\star P_\cN^\perp\subset P_B^\perp$. 
  We show its reverse inclusion.
  Let $q(\partial_\s)\in P_B^\perp$.
  By Lemma \ref{lem:staroperation} (iii), there exists $r(\partial_\s)\in\C[\partial_\s]$ such that $q(\partial_\s)=m(\partial_\z)\star r(\partial_\s)$. It suffices to show that $r(\partial_\s)\in P_\cN^\perp$.
  Let $f(\s)\in P_\cN$. By the assumption, we have $f(\s)=m(\s)g(\s)$ for some $g(\s)\in P_B$.
  By Lemma \ref{lem:Action}, we see that
  \begin{align}
    [r(\partial_\s)\bullet f(\s)]_{|\s=\0}&=[r(\partial_\s)\bullet (m(\s)g(\s))]_{|\s=\0}\\ 
    &=\left[\left(m(\partial_\z)\star r(\partial_s)\right)\bullet g(\s)\right]_{|\s=\0}\\
    &=\left[q(\partial_s)\bullet g(\s)\right]_{|\s=\0}=0.
  \end{align} 
  Hence we have the assertion.
\end{proof}

\begin{example}(cf.\,\cite[Examples 3.2 and 4.7]{Log})
  \label{ex:non_C-Mcase}
  Let $A=\begin{bmatrix} 1 & 1 & 1 & 1 \\ 0 & 1 & 3 & 4 \end{bmatrix}$ and let $\w=(3, 1, 0, 0)$.
  Then the reduced Gr\"{o}bner basis of $I_A$ is 
  \[
    \cG=\{\underline{\partial_{x_1}\partial_{x_3}^2}-\partial_{x_2}^2\partial_{x_4}, \underline{\partial_{x_2}\partial_{x_4}^2}-\partial_{x_3}^3, \underline{\partial_{x_1}^2\partial_{x_3}}-\partial_{x_2}^3,  \underline{\partial_{x_1}\partial_{x_4}}-\partial_{x_2}\partial_{x_3}\}.
  \]
  Here underlined terms are the leading ones.
  Thus we have
  \begin{equation}
    \ini_\w(I_A)=\langle \partial_{x_1}\partial_{x_3}^2, \partial_{x_2}\partial_{x_4}^2, \partial_{x_1}^2\partial_{x_3}, \partial_{x_1}\partial_{x_4} \rangle.
  \end{equation}
  Put
  \begin{equation}
    \begin{array}{ll}
      \g^{(1)}=(1, -2, 2, -1)^T, &\g^{(2)}=(0, 1, -3, 2)^T,\\[10pt]
      \g^{(3)}=(2, -3, 1, 0)^T, &\g^{(4)}=(1, -1, -1, 1)^T.
    \end{array}
  \end{equation}
  Let $\bbeta=(-2, -1)^T$, and let
  \begin{equation}
    B=(\g^{(1)}, \g^{(2)})=\begin{bmatrix} 1 & 0 \\ -2 & 1 \\ 2 & -3 \\ -1 & 2 \end{bmatrix}.
  \end{equation}
  Note that $\supp(B)=\{1,2,3,4\}$.
  Take $\v=(0,-2,-1,1)^T$ as a fake exponent.
  Then we have 
  \begin{align}
    \cN&=\{\{2\},\{3\},\{2,3\}=I_\0\},\\
    \cN^c&=\{\{1,2\},\{1,3\},\{1,4\},\{2,4\},\{1,2,4\},\{1,3,4\}\},\\ 
    K_\cN&=\emptyset.
  \end{align}
  Furthermore, we have
  \begin{equation}
    \begin{array}{ll}
    G^{(1)}=I_{-\g^{(1)}}\setminus I_\0=\{1\}, & G^{(2)}=I_{-\g^{(1)}}\setminus I_\0=\{4\}, \\[10pt]
    G^{(3)}=I_{-\g^{(3)}}\setminus I_\0=\{1\}, & G^{(4)}=I_{-\g^{(4)}}\setminus I_\0=\{1\}.
    \end{array}
  \end{equation}
  Thus the ideals $P_\cN$ and $P_B$ are
  \begin{align}
    P_\cN &=\langle (B\s)^{\{1,2\}}, (B\s)^{\{1,3\}}, (B\s)^{\{2,4\}}\rangle\\ 
    &=\langle s_1(-2s_1+s_2), s_1(2s_1-3s_2),(-2s_1+s_2)(-s_1+2s_2)\rangle\\ 
    &=\langle s_1^2, s_1s_2,s_2^2\rangle
  \end{align}
  and
  \begin{equation}
    P_B =\langle (B\s)^{\{1\}}, (B\s)^{\{4\}}\rangle=\langle s_1, -s_1+2s_2\rangle=\langle s_1, s_2\rangle,
  \end{equation}
  respectively.
  The orthogonal complements $P_\cN^\perp$ and $P_B^\perp$ are
  \begin{equation}
    P_\cN^\perp=\C 1+\C \partial_{s_1}+\C \partial_{s_2}
  \end{equation}
  and 
  \begin{equation}
    P_B^\perp=\C 1,
  \end{equation}
  respectively.
  In this case, note that  
  \begin{equation}
    m(\s)=(B\s)^{I_\0\setminus K_\cN}=(B\s)^{\{2,3\}}=(-2s_1+s_2)(2s_1-3s_2)\in P_\cN.
  \end{equation}
  Hence, by Proposition \ref{prop:Key2}, we have
  \begin{equation}
    m(\partial_\z)\star P_\cN^\perp=\{0\}.
  \end{equation}
\end{example}

\begin{lemma}\label{lem:homogeneous_q}
  Let $q(\z)\in\C[\z]:=\C[z_1,\ldots,z_h]$ be a homogeneous polynomial in indeterminates $\z$ of degree $r$.
  Then 
  \begin{equation}
    q(\partial_\s)\bullet \left(\frac{1}{r!}(\x B\s)^{r}\right)=q(\x B)=\Psi_B(q(\partial_\s))(\x).
  \end{equation}
  Here, 
  \begin{equation}
    \x B\s:=\sum_{j=1}^{n}\sum_{k=1}^{h}x_{j}b_{j}^{(k)}s_{k}
  \end{equation}
  denotes the quadratic form associated with $B$.
\end{lemma}
\begin{proof}
  We show the assertion by induction on $r$.
  In the case of $r=1$, the assertion is clear.
  Fix $r>1$ and assume that the assertion holds for any homogeneous polynomial of degree less than $r$. 
  Let $\mmu\in\N^h$ with $|\mmu|=r$ and $\mu_k>0$. 
  Then, by the chain rule and the induction hypothesis, we see that 
  \begin{align}
    {\partial_{s}}^{\mmu}\bullet \left(\frac{1}{r!}(\x B\s)^{r}\right)&={\partial_{s}}^{\mmu-\e_{k}}\bullet \left(\partial_{s_k}\bullet\left(\frac{1}{r!}(\x B\s)^{r}\right)\right)\\
    &=(\x B)_k \cdot{\partial_\s}^{\mmu-\e_k}\bullet\left(\frac{1}{(r-1)!} (\x B\s)^{r-1}\right)\\
    &=(\x B)_k (\x B)^{\mmu-\e_k}=(\x B)^{\mmu}=\Psi_B(\partial_\s^\mmu)(\x).
  \end{align}
\end{proof}
Two vector spaces $Q_\v^\perp$ and $P_B^\perp$ are related as follows.
\begin{theorem}
  \label{thm:Pperp Qperp}
  Let $\Psi_B$ be the homomorphism in \eqref{eqn:Psi}.
  Then, $P_B^\perp=\Psi_B^{-1}(Q_\v^\perp)$ and  $\dim_\C(P_B^\perp)\leq \dim_\C(Q_\v^\perp)$.
  Furthermore, if $B$ is a basis of $L$, then $\Psi_B(P_B^\perp)=Q_\v^\perp$ and $\dim_\C(Q_\v^\perp)=\dim_\C(P_B^\perp)$.
\end{theorem}
\begin{proof}
  Let $\deg(q(\z))=r$. Then, it follows from Lemma \ref{lem:homogeneous_q} that
  \begin{align}
    \partial_\x^{G^{(i)}}&\bullet q(\x B)=\partial_\x^{G^{(i)}}\bullet \left\{q(\partial_\s)\bullet \left(\frac{1}{r!}(\x B\s)^{r}\right)\right\}\\
    &=q(\partial_\s)\bullet\left\{\partial_\x^{G^{(i)}}\bullet \left(\frac{1}{r!}(\x B\s)^{r}\right)\right\}\\
    &=q(\partial_\s)\bullet\left\{\partial_\x^{G^{(i)}}\bullet \left(\frac{1}{r!}\sum_{\mmu\in \N^n; |\mmu|=r}\frac{r!}{\mmu!} \x^\mmu(B\s)^\mmu\right)\right\}\\
    &=q(\partial_\s)\bullet\left\{\sum_{\mmu\in \N^n; |\mmu|=r}\frac{\partial_\x^{G^{(i)}}\bullet\x^\mmu}{\mmu!} (B\s)^\mmu\right\}\\
    &=q(\partial_\s)\bullet\left\{\sum_{\mmu\in\N^n;|\mmu|=r, \supp(\mmu)\supset G^{(i)}}\frac{\x^{\mmu-\e_{G^{(i)}}}}{(\mmu-\e_{G^{(i)}})!}(B\s)^{\mmu}\right\}\\
    &=\sum_{\mmu\in\N^n;|\mmu|=r, \supp(\mmu)\supset G^{(i)}}\frac{\x^{\mmu-\e_{G^{(i)}}}}{(\mmu-\e_{G^{(i)}})!}\left\{q(\partial_\s)\bullet(B\s)^{\mmu}\right\}
  \end{align}
  for any $i$.
  Here $\e_{G^{(i)}}:=\sum_{j\in G^{(i)}}\e_j$ denotes the indicator vector of $G^{(i)}$.
  Since $\partial_\s^{\p}\bullet \s^\q=\p!\delta_{\p,\q}$ for any $\p,\q\in\N^{h}$,
  by Proposition \ref{prop3} we have  
  \begin{align}
    &q(\x B)\in Q_\v^\perp\iff \partial_\x^{G^{(i)}}\bullet q(\x B)=0\ \text{for all}\ i=1,\ldots,m\\
    \iff & q(\partial_\s)\bullet(B\s)^{\mmu}=0\\
    &\displaystyle \qquad \text{for all}\ i\ \text{and all}\ \mmu \in\N^n\ \text{with}\ |\mmu|=r, \supp(\mmu)\supset G^{(i)}\\
    \iff & q(\partial_\s)\in P_B^\perp.
  \end{align}
  Here, the second equivalence follows from the linear independence of the monomials $\x^\mmu$.
  The third equivalence follows from the linear independence of $B$, because it yields that the ideal $P_B$ is spanned as a vector space by the polynomials whose terms are of the form $(B\s)^\mmu$ with $\supp(\mmu)\supset G^{(i)}$ for some $i$.
  Thus we have $P_B^\perp=\Psi_B^{-1}(Q_\v^\perp)$. 
  Moreover, we have the inequality $\dim_\C(P_B^\perp)\leq \dim_\C(Q_\v^\perp)$ because $\Psi_B({P_B}^\perp)=\Psi_B\left(\Psi_B^{-1}(Q_\v^\perp)\right)\subset Q_\v^\perp$ and $\Psi_B$ is injective.

  Assume that $B$ is a basis of $L$. 
  Note that each $\x \g^{(k)}=\sum_{j=1}^{n}g_{j}^{(k)}x_j$ can be represented by a linear combination of $(\x B)_1, \ldots, (\x B)_h$.
  Let $f(\x)\in Q_\v^\perp$. Then, by Proposition \ref{prop3} and the above result, there exists $q(\partial_\s)\in P_B^\perp$ such that $f(\x)=q(\x B)$. Thus we have $f(\x)=\Psi_B(q(\partial_\s))(\x)$.
\end{proof}

\section{Fundamental systems of solutions}
In this section, we construct a fundamental system of series solutions with a given exponent to $M_A(\bbeta)$. 
We recall that the homogeneity of $A$ yields the regular holonomicity of $M_{A}(\bbeta)$.
This means that, for a fixed generic weight $\w$, the solution space to $M_A(\bbeta)$ has a basis consisting of canonical series with starting monomial $x^\v(\log \x)^\b$ for some exponent $\v$ and $\b\in\N^n$.
Note that each $x^\v(\log \x)^\b$ is derived as the initial monomial of a solution to the indicial ideal $\ind_\w(H_A(\bbeta))_\v$, or of an element of $Q_\v^\perp$.     
For the detail, see \cite[Sections 2.3, 2.4, and 2.5]{SST} and Proposition \ref{prop3}. 

Throughout this section, we assume that $B$ is a basis of $L$. 
Since $B$ satisfies Assumption \ref{ass3} (see Remark \ref{rem:rem_to_ass3}), we have the following homomorphisms by Propositions \ref{prop3}, \ref{prop:startingterm}, and Theorem \ref{thm:Pperp Qperp}:
\begin{equation}\label{eqn:homomorphism2}
\begin{array}{ccccc}
  P_\cN^\perp &\rightarrow & P_B^\perp &\simeq &\mathrm{Sol}(\find_\w(H_A(\bbeta))_\v),\\
  q(\partial_\s) & \mapsto & m(\partial_\z)\star q(\partial_\s) & \leftrightarrow & x^\v\Psi_B(m(\partial_\z)\star q(\partial_\s))(\log \x).
\end{array}
\end{equation}
Here, $\mathrm{Sol}(\find_\w(H_A(\bbeta))_\v)$ denotes the solution space of the fake indicial ideal $\find_\w(H_A(\bbeta))_\v$.

\begin{proposition}
  \label{prop:exponent}
  If $m(\s)\notin P_\cN$, then $\v$ is an exponent. 
\end{proposition}
\begin{proof}
Assume that $m(\s)\notin P_\cN$.
Then, by Proposition \ref{prop:Key2} (ii), there exists $q(\partial_\s)\in P_\cN^\perp$ such that $m(\partial_\z)\star q(\partial_\s)\ne 0$.
We see from Theorem \ref{thm:refinement} that 
\begin{align}
  (q(\partial_\s)&\bullet \widetilde{F}_{\cN}(\x,\s))_{|\s=\0}\\
  &=\sum_{\u\in L'}\left(q(\partial_\s)\bullet(m(\s)a_\u(\s)x^{\v+B\s+\u})\right)_{|\s=\0}\\
  &=(q(\partial_\s)\bullet (m(\s)x^{\v+B\s})_{|\s=\0}\\
  &\qquad\qquad +\sum_{\u\in L'\setminus\{0\}}\left(q(\partial_\s)\bullet(m(\s)a_\u(\s)x^{\v+B\s+\u})\right)_{|\s=\0}
\end{align}
is a solution to $M_A(\bbeta)$.
By Proposition \ref{prop:startingterm}, we have
\begin{equation}
  (q(\partial_\s)\bullet (m(\s)x^{\v+B\s})_{|\s=\0}
  =x^\v\Psi_B(m(\partial_\z)\star q(\partial_\s))(\log \x), 
\end{equation}
hence this solution has a \textit{non-zero} starting term.
Hence $\v$ is an exponent.
\end{proof}

\begin{example}[Continuation of Example \ref{ex:non_C-Mcase}]
  \label{ex:non_C-Mcase_(2)}
  Let $A$ and $\v$ be the ones in Example \ref{ex:non_C-Mcase}.
  Recall that $m(\s)\in P_\cN$, which is a necessary condition for the fake exponent $\v$ not to be an exponent.
  We see that $\v$ is not an exponent from the following calculation.
  Note that
  \begin{equation}
    \theta_1-2\theta_3-3\theta_4+1\in \langle A\theta_\x-\bbeta\rangle.
  \end{equation}
  Then we have 
  \begin{align}
  0&\equiv
  \theta_1\theta_3(\theta_1-2\theta_3-3\theta_4+1)\\
  &=
  \theta_1^2\theta_3-2\theta_1\theta_3^2-3\theta_1\theta_3\theta_4+\theta_1\theta_3\\
  &=
  \theta_1(\theta_1-1)\theta_3-2\theta_1\theta_3(\theta_3-1)-3\theta_1\theta_3\theta_4\\
  &=
  x_1^2x_3\partial_{x_1}^2\partial_{x_3}-2x_1x_3^2\partial_{x_1}\partial_{x_3}^2-3x_1x_3x_4
  \partial_{x_1}\partial_{x_3}\partial_{x_4}\\
  &\equiv
  x_1^2x_3\partial_{x_2}^3-2x_1x_3^2\partial_{x_2}^2\partial_{x_4}-3x_1x_3x_4
  \partial_{x_2}\partial_{x_3}^2
  \end{align}
  modulo $H_A(\bbeta)$.
  Hence
  \begin{equation}
  x_1^2x_3\partial_{x_2}^3-2x_1x_3^2\partial_{x_2}^2\partial_{x_4}-3x_1x_3x_4
  \partial_{x_2}\partial_{x_3}^2\in H_A(\bbeta),
  \end{equation}
  and
  \begin{equation}
    x_1x_3^2\partial_{x_2}^2\partial_{x_4}\in \ini_{(-\w,\w)}(H_A(\beta)).
  \end{equation}
  Since
  \begin{equation}
    x_1x_3^2\partial_{x_2}^2\partial_{x_4}\bullet x^\v=
    x_1x_3^2\partial_{x_2}^2\partial_{x_4}\bullet x_2^{-2}x_3^{-1}x_4 \neq 0,
  \end{equation}
  $\v$ is not an exponent.
\end{example}

\begin{corollary}
  Assume that $B$ is a basis of $L$. 
  If $|I\cup J|>|I_\0|$ for any $I\in \cN$ and $J\in \cN^c$, then $\v$ is an exponent.  
\end{corollary}
\begin{proof}
For any $I\in\cN$ and $J\in \cN^c$, we see that 
\begin{equation}
  |I\cup J\setminus K_\cN|=|I\cup J|-|K_\cN|>|I_\0|-|K_\cN|=|I_\0\setminus K_\cN|
\end{equation}
because both of $I$ and $I_\0$ contain $K_\cN$.
Since the degree of $m(\s)=(B\s)^{I_\0\setminus K_\cN}$ is less than that of any $(B\s)^{I\cup J\setminus \cN}$, $m(\s)$ cannot belong to $P_\cN$. 
By Proposition \ref{prop:exponent}, $\v$ is an exponent.
\end{proof}

\begin{theorem}
  \label{thm:PBKandQperp}
  Assume that $B$ is a basis of $L$, and that $P_\cN=m(\s)\cdot P_B$.
  Then $\v$ is an exponent, and the set 
  \begin{equation}
    \{(q(\partial_{\s})\bullet \widetilde{F_{\cN}}(x,\s))_{|\s=\0}\,|\,q(\partial_\s)\in P_\cN^\perp\}.
  \end{equation}
  spans the space of series solutions in the direction of $\w$ to $M_{A}(\bbeta)$ with exponent $\v$.
  In particular, for $q(\partial_\s)\in P_\cN^\perp$, the solution $(q(\partial_{\s})\bullet \widetilde{F}_{\cN}(x,\s))_{|\s=\0}$ has the starting term $x^\v\Psi_B(m(\partial_\z)\star q(\partial_\s))(\log \x)$.
  \end{theorem}
  \begin{proof}
    By definition, $P_B\ne\C[\s]$. 
    It follows from the assumption that $m(\s)\notin P_\cN$.
    Hence, by Proposition \ref{prop:exponent}, $\v$ is an exponent.
    Moreover, by Proposition \ref{prop:Key2}, the homomorphism \eqref{eqn:homomorphism2}  is surjective.
    Hence we have
    \begin{align}
      \dim_\C (P_\cN^\perp)&\geq \dim_\C(P_B^\perp)=\dim_\C(Q_\v^\perp)=\dim_\C(\mathrm{Sol}(\find_\w(H_A(\bbeta))_\v))\\
      &\geq \dim_\C(\mathrm{Sol}(\ind_\w(H_A(\bbeta))_\v)).
    \end{align}

    By the regularity of $M_A(\bbeta)$, the dimension of the space of series solutions with the exponent $\v$ coincides with $\dim_\C(\mathrm{Sol}(\ind_\w(H_A(\bbeta))_\v))$. 
    Hence, by Theorem \ref{thm:refinement}, we have the former half of the assertion.

    The latter half of the assertion follows from Proposition \ref{prop:startingterm}.
    \end{proof}

\begin{example}[Continuation of Examples \ref{ex:SST-ex3.5.2_(1)}, \ref{ex:SST-ex3.5.2_(2)}, and \ref{ex:SST-ex3.5.2_(3)}]
\label{ex:SST-ex3.5.2_(4)}
  Consider the case where $B=\{\g^{(1)}, \g^{(2)}\}$.
  Then, $B$ satisfies the assumption in Theorem \ref{thm:PBKandQperp}.
  Recall that $m(\s)=(B\s)^{\emptyset}=1$, and
  \begin{equation}
    P_B=\langle (B\s)^{G^{(1)}}, (B\s)^{G^{(2)}}\rangle =\langle s_1^2, s_2^2\rangle=P_\cN.
  \end{equation}
  Hence, we see by Proposition \ref{prop:exponent} that $\v$ is an exponent, and that  
  \begin{equation}
    \{1, \partial_{s_1}, \partial_{s_2}, \partial_{s_1}\partial_{s_2}\}
  \end{equation}
  is a basis of $P_B^\perp$.
  Hence $x^\v f(\log \x)$ is a solution to 
  $\find_\w(H_A(\bbeta))_\v$ 
  if and only if
  \begin{equation}
  f\in \langle 1, \x\g^{(1)}, \x\g^{(2)}, (\x\g^{(1)})\cdot (\x\g^{(2)})\rangle_\C.
  \end{equation}
  By the uniqueness of an exponent, the above space coincides with the space of solutions to $M_A(\bbeta)$. 
  Note that the holonomic rank of $M_A(\bbeta)$ is four (cf.\,\cite[Example 3.5.2]{SST}). 
\end{example}

\begin{example}\cite[Examples 3.6.3, 3.6.11, 3.6.16]{SST}
\label{ex:SST-ex3.6.3}
Let $d=3$, $n=9$ and
\begin{equation}
A=
\begin{bmatrix}
1 & 1 & 1 & 1 & 1 & 1 & 1 & 1 & 1\\
0 & 1 & 2 & 0 & 1 & 2 & 0 & 1 & 2\\
0 & 0 & 0 & 1 & 1 & 1 & 2 & 2 & 2
\end{bmatrix}.
\end{equation}
Let $\w=(2,0,0,0,-1,0, 0,0, 2)$,
$\bbeta=(1,1,1)^T=\a_5$.
Consider an exponent $\v=(0,0,0,0,1,0,0,0,0)^T$.
Then $K_{\cN_B}=I_\0=\emptyset$.
The reduced Gr\"{o}bner basis consists of the following twenty vectors 
\begin{align}
  &\{\g^{(1)}:=(0,1,-1,0,-1,1,0,0,0)^T,\g^{(2)}:=(0,0,0,1,-1,0,-1,1,0)^T,\\
  &\quad \g^{(3)}:=(0,0,0,1,-2,1,0,0,0)^T,\g^{(4)}:=(0,1,0,0,-2,0,0,1,0)^T,\\
  &\quad \g^{(5)}:=(1,-1,0,-1,1,0,0,0,0),\g^{(6)}:=(0,0,0,0,1,-1,0,-1,1),\\
  &\quad \g^{(7)},\ldots,\g^{(20)}\}.
\end{align}
Hence we have 
\begin{align}
  \{&G^{(i)}\,|\,i=1,\ldots,20\}\\
  &=\{G^{(1)}=\{2,6\},G^{(2)}=\{4,8\},G^{(3)}=\{4,6\},G^{(4)}=\{2,8\},\\
  &\qquad G^{(5)}=\{1\},G^{(6)}=\{9\},\{1,8\},\{2,7\},\{1,3\},\{1,6\},\{3,4\},\\
  &\qquad\{1,9\},\{3,7\},\{4,9\},\{6,7\},\{7,9\},\{2,9\},\{3,8\},\{3,9\},\{1,7\}\}.
\end{align}

Let 
\begin{equation}
  B=[\g^{(1)},\g^{(2)},\g^{(3)},\g^{(4)},\g^{(5)},\g^{(6)}]=\begin{bmatrix} 0 & 0 & 0 & 0 & 1 & 0\\
    1 & 0 & 0 & 1 & -1 & 0\\
    -1 & 0 & 0 & 0 & 0 & 0\\
    0 & 1 & 1 & 0 & -1 & 0 \\
    -1 & -1 & -2 & -2 & 1 & 1\\
    1 & 0 & 1 & 0 & 0 & -1\\
    0 & -1 & 0 & 0 & 0 & 0\\
    0 & 1 & 0 & 1 & 0 & -1\\
    0 & 0 & 0 & 0 & 0 & 1
  \end{bmatrix}.
\end{equation}
Then 
\begin{align}
  P_{\cN}&=P_B\\
  &=\langle (s_1+s_4-s_5)(s_1+s_3-s_6),(s_2+s_3-s_5)(s_2+s_4-s_6),\\
  &\qquad (s_2+s_3-s_5)(s_1+s_3-s_6),(s_1+s_4-s_5)(s_2+s_4-s_6),\\
  &\qquad s_5, s_6,s_2(s_1+s_4-s_5),s_1(s_2+s_3-s_5),\\
  &\qquad s_1 s_2,s_2(s_1+s_3-s_6),s_1(s_2+s_4-s_6)\rangle\\
  &=\langle s_1 s_2, s_1 s_3, s_1 s_4, s_2 s_3, s_2 s_4, s_3^2, s_4^2, s_1^2+s_3 s_4, s_2^2+s_3 s_4, s_5, s_6\rangle,
\end{align}
where the last generator set gives the reduced Gr\"{o}bner basis with respect to the lexicographic order $<$ with $s_1>s_2>s_3>s_4>s_5>s_6$.
We see that
\begin{equation}
  \{\nnu\in\N^{6}\,|\,\s^\nnu\notin\ini_<(P_B)\}=\{\0,\e_1,\e_2,\e_3,\e_4,\e_3+\e_4\},
\end{equation}
and hence we immediately have
\begin{equation}
  q_\0=1, q_{\e_1}=\partial_{s_1}, q_{\e_2}=\partial_{s_2}, q_{\e_3}=\partial_{s_3}, q_{\e_4}=\partial_{s_4}.
\end{equation}
As to the last generator $q_{\e_3+e_4}$, since 
\begin{equation}
  c_{\mmu,\e_3+\e_4}=\begin{cases} -1 & (\text{if}\ \mmu= 2\e_1, 2\e_2)\\ 0 & (\text{otherwise})
\end{cases},
\end{equation}
we have
\begin{equation}
  q_{\e_3+\e_4}=\partial_{s_3}\partial_{s_4}+\sum_{\mmu;|\mmu|=2}c_{\mmu,\e_3+\e_4}\frac{1}{\nnu!}\partial_\s^\mmu=\partial_{s_3}\partial_{s_4}-\frac{1}{2}\partial_{s_1}^2-\dfrac{1}{2}\partial_{s_2}^2.
\end{equation}
Hence a $\C$-basis of $Q_\v^\perp$ is given by 
\begin{equation}
\{1, \x\g^{(1)}, \x\g^{(2)} ,\x\g^{(3)}, \x\g^{(4)}, 
(\x\g^{(3)})\cdot(\x\g^{(4)})-\frac{1}{2} (\x\g^{(1)})^2-\frac{1}{2} (\x\g^{(2)})^2\}.
\end{equation}
\end{example}


\section{Aomoto-Gel'fand systems}
In this section, let
\begin{equation}
A=\{ \a_{i,j}\,|\, 1\leq i\leq m,\, m+1\leq j\leq m+l\},
\end{equation}
where
$\a_{i,j}=\e_i+\e_j$, and $\{ \e_1,\ldots, \e_{l+m}\}$
is the standard basis of $\Z^{l+m}$.

Then $\Z A=\{ \a\in \Z^{l+m}\,|\, \sum_{i=1}^ma_i=\sum_{j=m+1}^{m+l}a_j\}$,
$\rank(A)=m+l-1$, and
$\rank(L)=ml-(m+l-1)=(m-1)(l-1)$,
where
\begin{equation}
L=\{
[c_{ij}]_{1\leq i\leq m,\, m+1\leq j\leq m+l}\in M_{m\times l}(\Z)
\,|\,
\sum_{i,j}c_{ij}\a_{ij}=\0
\}.
\end{equation}
Since $A$ is normal, (that is, $\N A=\Z A\cap \R_{\geq 0}A$), $I_A$ is a Cohen-Macaulay ideal and hence
$\rank(M_A(\bbeta))=\vol(A)$ for any $\bbeta$ (see \cite{GZK2}).

Take a weight vector $\w$ satisfying
$w_{i,j}>w_{p,q}$ whenever $(i, j)\ne (p,q)$, $i\leq p$, and $j\leq q$.

Then the reduced Gr\"{o}bner basis of $I_A$ with respect to $\w$ equals
\begin{equation}
G:=
\{
\underline{\partial^{(\g^{(i,j)}_{(p,q)})_+}}
-
\partial^{(\g^{(i,j)}_{(p,q)})_-}\,|\,
i<p,\, j<q
\},
\end{equation}
and
\begin{equation}
\ini_\w(I_A)=
\langle
\partial_{i,j}\partial_{p,q}\,|\, i<p,\, j<q
\rangle,
\end{equation}
where 
$\g^{(i,j)}_{(p,q)}:=E_{i,j}+E_{p,q}-E_{i,q}-E_{p,j}\in L$
and $E_{i,j}$ are matrix units.

The weight $\w$ induces a staircase regular triangulation,
which is unimodular
(cf. \cite[Example 8.12]{Sturmfels-GB});
for example, let $m=2, l=4$, the standard pairs are
\begin{equation}
\begin{bmatrix}
* & * & * & *\\
* & 0 & 0 & 0
\end{bmatrix},
\quad
\begin{bmatrix}
0 & * & * & *\\
* & * & 0 & 0
\end{bmatrix},
\quad
\begin{bmatrix}
0 & 0 & * & *\\
* & * & * & 0
\end{bmatrix},
\quad
\begin{bmatrix}
0 & 0 & 0 & *\\
* & * & * & *
\end{bmatrix},
\end{equation}
where let
the row-numbers be $1,\ldots, m$ and
the column-numbers $m+1,\ldots, m+l$.

For general $l,m$,
the standard pairs correspond to the paths
from the southwest corner to the northeast corner
going only northward or eastward.
In this way, we see
\begin{equation}
\vol(A)=
\binom{m+l-2}{m-1}.
\end{equation}

Let
\begin{equation}
B:=
\{ \b^{(i,j)}:=\g^{(i,j)}_{(i+1, j+1)}\,|\,
1\leq i<m,\, m+1\leq j< m+l\}.
\end{equation}
Then $B$ is a basis of $L$ and $\supp(B)=\{1,\ldots,m+l\}$, hence $B$ satisfies Assumption \ref{ass3}.
Let $\s=(s_{(i,j)})_{1\leq i<m,m+1\leq j<m+l}$ be indeterminates such that $s_{(i,j)}$ corresponds to $\b^{(i,j)}$.
For convenience, set $s_{(i,j)}:=0$ unless $(i,j)\in \{1,\ldots,m-1\}\times\{m+1,\ldots,m+l-1\}$. 
Then 
\begin{align}
  (B\s)_{(\mu,\nu)}&=\sum_{\substack{1\leq i<m\\m+1\leq j<m+l}}s_{(i,j)}(\g_{(i+1, j+1)}^{(i,j)})_{(\mu,\nu)}\\
  &=s_{(\mu,\nu)}-s_{(\mu,\nu-1)}-s_{(\mu-1,\nu)}+s_{(\mu-1,\nu-1)}.
\end{align}

\begin{lemma}\label{lem:Aomoto:exponent}
Let $\bbeta=\0$. 
Then $\v=\0$ is a unique exponent.
\end{lemma}

\begin{proof}
Since $\ini_\w(I_A)$ is square-free, every fake exponent is an exponent by Theorem 3.6.6 in \cite{SST}.

Let $\v$ be an exponent.
Then there exists a standard pair $(\a, \sigma)=(\0,\sigma)$ corresponding to $\v$ such that 
\begin{equation}
  v_j=0\quad (j\notin \sigma),\qquad A\v=\bbeta=\0.
\end{equation}
Since the submatrix $A_\sigma=(\a_j)_{j\in\sigma}$ is invertible and satisfies $A_\sigma \v_\sigma=\0$, we have $\v=\0$.
\end{proof}

From now on, let $\bbeta=\0$ and $\v=\0$.
Hence $I_\0=\emptyset=K_{\cN_B}$.

\begin{lemma}
\label{lem:Aomoto:J}
\begin{enumerate}
\item[(i)]
$\{ (ij), (pq)\}\in \cN^c=\NS_\w(\v)^c$ for $i<p,\, j<q$.
\item[(ii))]
Let $J\in \cN^c=\NS_\w(\v)^c$.
Then there exist $i<p,\, j<q$
such that $J\supset \{ (ij), (pq)\}$.
\end{enumerate}
\end{lemma}

\begin{proof}
(i)\
This follows from
$G^{(ij)}_{(pq)}=\nsupp(\v-\g^{(ij)}_{(pq)})=\nsupp(-\g^{(ij)}_{(pq)})=
\{ (ij), (pq)\}.
$

(ii)\ 
This is immediate from Proposition \ref{prop:Assump(2)}.
\end{proof}

\begin{proposition}
\label{prop:Aomoto:P_B}
\begin{equation}
P_B=\langle (B\s)^{\{(i,j),(p,q)\}}\,|\,i<p, j<q\rangle
=
\langle
s_{(i,j)}s_{(p,q)}\,|\,
i\leq p,\, j\leq q
\rangle.
\end{equation}
\end{proposition}

\begin{proof}
  The first equality follows from Lemma \ref{lem:Aomoto:J}.
We show the second equality.
Since, for any $i$, $j$, $p$, $q$ with $i<p$ and $j<q$, 
\begin{align}
  (B\s)^{\{(i,j),(p,q)\}}&=(s_{(i,j)}-s_{(i,j-1)}-s_{(i-1,j)}+s_{(i-1,j-1)})\\
  &\quad\times(s_{(p,q)}-s_{(p,q-1)}-s_{(p-1,q)}+s_{(p-1,q-1)})\in \mathrm{RHS},
\end{align}
we only need to show the reverse inclusion. 
We introduce the total order $<$ on $\cX :=\{((i,j), (p,q))\,|\,1\leq i\leq p< m, m+1\leq j\leq q< m+l\}$ defined by
\begin{align}
  ((i, j), (p,q))&<((i', j'), (p',q'))\\
  &\iff \begin{cases}
    i<i',\\
    \text{or}\ [i=i'\ \text{and}\ j<j'],\\
    \text{or}\ [(i,j)=(i',j')\ \text{and}\ p>p'],\\
    \text{or}\ [(i,j,p)=(i',j',p')\ \text{and}\ q>q'].
  \end{cases}
\end{align}
We prove by induction on the totally ordered set $(\cX,<)$.

First, for the minimum element $((1,m+1), (m-1, m+l-1))$, we have
\begin{equation}
  s_{(1,m+1)}s_{(m-1,m+l-1)}=(B\s)^{\{(1,m+1),(m,m+l)\}}\in \mathrm{LHS}.
\end{equation}

Next, fix $((i,j),(p,q))\ne ((1,m+1), (m-1, m+l-1))$ with $1\leq i\leq p<m$ and $m+1\leq j\leq q<m+l$.
Suppose that $s_{(i',j')}s_{(p',q')}\in\mathrm{LHS}$ for all $((i',j'),(p',q'))<((i,j),(p,q))$.
Then the leading monomial in 
\begin{align}
\left(\mathrm{LHS}\ni\right)\ (B\s)^{\{(i,j),(p+1,q+1)\}}
&=(s_{(i,j)}-s_{(i,j-1)}-s_{(i-1,j)}+s_{(i-1,j-1)})\\
&\quad \times(s_{(p+1,q+1)}-s_{(p+1,q)}-s_{(p,q+1)}+s_{(p,q)})
\end{align}
is $s_{(i,j)}s_{(p,q)}$.
It follows from the induction hypothesis that $s_{(i,j)}s_{(p,q)}\in \mathrm{LHS}$.
Hence we have the assertion.
\end{proof}

\begin{corollary}
  \label{cor:Aomoto}
  \begin{align}
    P_B^\perp&=\left\langle \partial_\s^\p\,|\, \s^{\p}\notin\langle
    s_{(i,j)}s_{(p,q)}\,|\,
    i\leq p,\, j\leq q
    \rangle\right\rangle_\C\\
    &=\langle \partial_{s_{(i_1,j_1)}}\cdots\partial_{s_{(i_r,j_r)}}\,|\,i_1<\cdots<i_r, j_1>\cdots>j_r\rangle_\C.
  \end{align}
Furthermore, 
\begin{equation}
\dim_\C(P_B^\perp)
=\sum_{r=0}^{\min\{ l-1, m-1\}}
\binom{l-1}{r}\binom{m-1}{r}
=\binom{l+m-2}{m-1}.
\end{equation}
\end{corollary}

\begin{proof}
The first part follows from Proposition \ref{prop:Aomoto:P_B}.
For the second part, compare the coefficients of $z^{m-1}$
in the following:
\begin{align}
\sum_{k=0}^{l+m-2}
\binom{l+m-2}{k}z^k
&=
(1+z)^{l+m-2}
=(1+z)^{l-1}(1+z)^{m-1}\\
&=
(\sum_{p=0}^{l-1}
\binom{l-1}{p}z^p)
(\sum_{q=0}^{m-1}
\binom{m-1}{q}z^{m-1-q})\\
&=
\sum_{p,q}\binom{l-1}{p}\binom{m-1}{q}z^{m-1-q+p}.
\end{align}
\end{proof}

\begin{corollary}
\begin{equation}
\{
((\partial_{s_{(i_1,j_1)}}\cdots\partial_{s_{(i_r,j_r)}})\bullet F_{\cN}(\x,\s))_{|\s=\0}\,|\, i_1<\cdots<i_r, j_1>\cdots>j_r
\}
\end{equation}
forms a fundamental system of solutions
to $M_A(\0)$.
\end{corollary}

\begin{proposition}
\begin{equation}
Q_\0^\perp
=
\left\langle
\prod_{k=1}^r
(\x \g^{(i_k,j_k)}_{(i_k+1,j_k+1)})\,\middle|\,
\begin{array}{l}
i_1<i_2<\cdots< i_r\\
j_1>j_2>\cdots > j_r
\end{array}
\right\rangle_\C.
\end{equation}
\end{proposition}

\begin{proof}
  This is immediate from Theorem \ref{thm:Pperp Qperp} and Corollary \ref{cor:Aomoto}.
\end{proof}

\begin{example}
Let $m=2$. This case corresponds to the Lauricella's $F_D$
(e.g. see \cite[\S 3.1.3]{Aomoto-Kita}).
Then
$\vol(A)=\binom{l+m-2}{m-1}=l$,
and
\begin{equation}
P_B^\perp
=
\langle
1, \partial_{s_{(1,3)}}, \partial_{s_{(1,4)}},\ldots, \partial_{s_{(1,l+1)}}
\rangle_\C.
\end{equation}
\end{example}

\begin{example}
Let $l=m=3$. 
Then
$\vol(A)=\binom{l+m-2}{m-1}=\binom{4}{2}=6$,
and
\begin{equation}
P_B^\perp
=
\langle
1, \partial_{s_{(1,4)}}, \partial_{s_{(1,5)}}, \partial_{s_{(2,4)}}, \partial_{s_{(2,5)}},
\partial_{s_{(1,5)}}\partial_{s_{(2,4)}}
\rangle_\C.
\end{equation}
\end{example}

\section{Lauricella's $F_C$}

Let 
\begin{equation}
A:=\{ \a_i:=\e_0+\e_i,\, \a_{-i}:=\e_0-\e_i\,|\, i=1,2,\ldots, m\}.
\end{equation}
In this case, the $A$-hypergeometric systems correspond to 
Lauricella's $F_C$ \cite{Contiguity}.
Then
\begin{equation}
\Z A=\{
\a\in \Z^{m+1}\,|\, \sum_{i=0}^m a_i\in 2\Z\},
\end{equation}
and
\begin{equation}
L=
\{
\l\in \Z^{2m}\,|\,
\sum_{i=\pm 1,\ldots, \pm m}l_i=0,\,
l_i-l_{-i}=0\ \, (1\leq i\leq m)
\}.
\end{equation}
We have
$\rank(L)=2m-(m+1)=m-1$.

Take a weight $\w$ so that
\begin{equation}
w_1+w_{-1}> w_2+w_{-2}>\cdots> w_m+w_{-m}.
\end{equation}
Then
\begin{equation}
\ini_\w(I_A)
=\langle
\partial_1\partial_{-1},
\partial_2\partial_{-2},
\ldots,
\partial_{m-1}\partial_{-(m-1)}
\rangle,
\end{equation}
and the reduced Gr\"{o}bner basis $G$ is given by
\begin{equation}
G=\{ \g^{(1)}, \g^{(2)}, \ldots, \g^{(m-1)}\},
\end{equation}
where
$\g^{(i)}=\e_i+\e_{-i}-\e_m-\e_{-m}$.
We see that $G$ is a basis of the free $\Z$-module $L$. Set $B:=G$, then $B$ satisfies Assumption \ref{ass3}.
Note that $\supp(B)=\{\pm 1,\ldots,\pm m\}$.
Let $\s=(s_i)_{1\leq i\leq m-1}$ be indeterminates such that $s_i$ corresponds to $\b^{(i)}=\g^{(i)}$.
The standard pairs are
pairs of $*$-place
$\{ \epsilon(i)i\,|\, i\in [1,m-1]\}\cup \{ \pm m\}$
$(\epsilon : [1, m-1]\to \{ \pm 1\})$
and $0$-place its complement.

Hence $\vol(A)=2^{m-1}$.

\begin{lemma}
Let $\bbeta=\0$. 
Then $\v=\0$ is a unique exponent.
\end{lemma}

\begin{proof}
  The proof is similar to Lemma \ref{lem:Aomoto:exponent}.
\end{proof}

\begin{proposition}\label{prop:LauricellaFC1}
  \begin{equation}
    P_B=\langle (B\s)^{\{\pm i\}}\,|\,i=1,\ldots, m-1\rangle
    =
    \langle
    s_i^2\,|\,
    1\leq i\leq m-1
    \rangle.
  \end{equation}
  \end{proposition}
  
  \begin{proof}
  We have $G^{(i)}=\{ \pm i\}$ and 
  \begin{equation}
    (B\s)^{\{\pm i\}}=\left(\sum_{\nu=1}^{m-1}s_\nu\g^{(\nu)}_{+i}\right)\left(\sum_{\nu=1}^{m-1}s_\nu\g^{(\nu)}_{-i}\right)=s_i^2.
  \end{equation}
  \end{proof}
\begin{proposition}
  \begin{align}
    P_B^\perp&=\left\langle \partial_\s^\p\,|\,\s^\p\notin \langle
    s_i^2\,|\,
    1\leq i\leq m-1
    \rangle\right\rangle_\C\\
    &=\langle \partial_\s^I=\prod_{i\in I}\partial_{s_i} \,|\, I\subset \{1,\ldots, m-1\}\rangle_\C,
  \end{align}
  and
  \begin{equation}
  Q_\0^\perp=
  \langle
  \prod_{i\in I}
  (\x \g^{(i)})\,|\, I\subset \{1,\ldots,m-1\}\rangle_\C.
  \end{equation}
  Furthermore,
  \begin{equation}
    \dim_\C(Q_\0^\perp)=2^{m-1}.
  \end{equation}
\end{proposition}
\begin{proof}
  This is immediate from Lemma \ref{lem:Basis_of_Pperp}, Theorem \ref{thm:Pperp Qperp}, and Proposition \ref{prop:LauricellaFC1}.
  \end{proof}

\begin{corollary}
$\{ ((\prod_{i\in I}\partial_{s_i})\bullet F_{\cN}(\x,\s))_{|\s=\0}\,|\, I\subset [1,m-1]\}$
forms a basis of solutions of $M_A(\0)$.
\end{corollary}


\end{document}